\documentclass{amsart}
\usepackage{graphicx,color,amsfonts,amssymb,amsmath,dsfont,enumitem,hyperref} 

\title{Renewal structure of the tree builder random walk}
\author{Rodrigo Ribeiro}
\address{University of Denver, Colorado, USA.}
\email{rodrigo.ribeiro@du.edu}

\date{July 2023}
\numberwithin{equation}{section}

\newtheorem{theorem}{Theorem}
\newtheorem{lemma}{Lemma}

\newtheorem{proposition}{Proposition}
\newtheorem*{remark}{Remark}
\newtheorem{claim}{Claim}
\newtheorem*{lemma*}{Lemma}

\newenvironment{claimproof}[1][Proof of Claim]{%
  \proof[#1]\mbox{}\\*}{\endproof}

\newcommand{\pspace}{\mathcal{P}(\mathbb{N})}
\newcommand{\dtv}[2]{\| #1-#2\|_{TV}}
\newcommand{\dist}[2]{{\rm dist}(#1,#2)}

\newcommand{\degr}[2]{{\rm deg}(#2)}

\newcommand{\renew}[1]{\tau_{#1}}
\newcommand{\lstop}{\eta}
\newcommand{\rtime}{\widetilde{H}}

\newcommand{\El}[1]{E_{#1}}

\newcommand{\Pl}[1]{P_{#1}}
\newcommand{\cPl}{\mathbb{P}}

\newcommand{\Lbrace}{\left \lbrace}
\newcommand{\Rbrace}{\right \rbrace}

\begin{document}

\begin{abstract} In this paper, we study a class of random walks that build their own tree. At each step, the walker attaches a random number of leaves to its current position. The model can be seen as a subclass of the Random Walk in Changing Environments (RWCE) introduced by G. Amir, I. Benjamini, O. Gurel-Gurevich and G. Kozma. We develop a renewal framework for the process analogous to that established by A-S. Sznitman and M. Zerner in the context of RWRE. This provides a more robust foundation for analyzing the model. As a result of our renewal framework, we estabilish several limit theorems for the walker's distance, which include the Strong Law of Large Numbers (SLLN), the Law of the Iterated Logarithm (LIL), and the Invariance Principle, under an i.i.d. hypothesis for the walker's leaf-adding mechanism. Further, we show that the limit speed defined by the SLLN is a continuous function over the space of probability distributions on $\mathbb{N}$.

\end{abstract}
\maketitle




\section{Introduction}

Recent years have seen work on random walks on graphs that change over time \cite{avena2016mixing,collevecchio2006,collevecchio2017speed,cotar2017edge,Huang_EvolvingSets,disertori2015transience,durrett2002once,kious2017phase,peres2015random}. In this paper, we consider a family of models in which at each step the walker attaches a random number of vertices to its position.

This model and variants thereof have been studied in a series of recent papers \cite{EIR21,FIORR21,IRVZ22}. They are perhaps the simplest processes where the walker can radically alter the structure of the graph. This stands in contrast with other models, such as random walks in random and changing environments, where the structure of the graph may be random, and maybe even evolve over time, but in a much more constrained way. The model has also been studied as a model for a ``self-creating polymer'' see \cite{DS22} and \S \ref{sss:cw} below. 

Up to now, the so-called Tree Builder Random Walker (TBRW) has been studied via ad-hoc methods. The main technical contribution of this paper is a new result that shows that this model has an intrinsic renewal structure. A formal definition of this structure is given in \S \ref{s:renewalstructure}. For now we note that renewal type results naturally lead to many results, as is known from other models. In particular, the following new results are obtained in the present paper: Strong Law of Large Numbers, Law of the Iterated Logarithm, Central Limit Theorem and Invariance Principle for the distance of the walker from the root of its tree.

Under mild assumptions on the leaf-adding mechanism of the TBRW, the SLLN states that it moves away from the root at a well-defined linear speed. Whereas the other limit theorems provide a finer understanding of the fluctuations around this limiting speed. We also show that the limiting speed, when seen as a function from the space of probability distributions, is continuous.

To facilitate a comprehensive discussion of our results, we will first define the model in the following section.

\subsection{The Tree Builder Random Walk (TBRW)}
The Tree Builder Random Walk (TBRW) was introduced at \cite{cannings2013} with a restart feature, that is, the walker was restarted after adding a new vertex to the graph. Later, in \cite{iacobelli2016} an instance of it in which the walker adds a vertex after $s$ taking steps and without reseting its position was introduced. In \cite{IRVZ22}, the authors introduced the name TBRW and added an extra layer of randomness letting the walker to add a random number of vertices at each step.

The TBRW is a discrete-time stochastic process $\{(T_n, X_n)\}_{n\geq 0}$ where $T_n$ is a rooted tree and $X_n$, a vertex of the tree, is the position of the walker on $T_n$ at time~$n$. The dynamics of the TBRW model relies upon a sequence of random variables $\{\xi_n\}_{n \ge 0}$ supported on $\mathbb{N}$\footnote{Throughout this paper, we will consider $0$ as a natural number.}, called the {\it leaf process}. And an initial state $(T_0,x_0)$ where $T_0$ is a locally finite tree and $x_0$ is a vertex of $T_0$.

The model then evolves inductively, that is, we obtain $(T_n,X_n)$ from $(T_{n-1},X_{n-1})$ as follows: firstly, $T_n$ is obtained by adding $\xi_n$ leaves to $X_{n-1}$. In the case $\xi_n = 0$, $T_n = T_{n-1}$. Secondly, the random walk takes a step on $T_n$ by jumping to an uniformly chosen neighbor of $X_{n-1}$ in the possibly updated $T_n$.

Given that the TBRW represents a novel model in this degree of generality, we will lay the groundwork by delineating certain conventions and intricacies of the model. Subsequently, in Sections \ref{sss:graph}, \ref{sss:rwce}, and \ref{sss:cw}, we will discuss the various perspectives of the TBRW. In particular, its application to real-world phenomena and its connection to other models within the random walk literature.

\subsubsection{Notation and Conventions}\label{sss:notation}
In general, we will let $\Pl{T,x;\xi}$ and $\El{T,x;\xi}$ denote respectively the law and expectation with respect to a TBRW starting from the initial state~$(T,x)$, which will always be finite, and a leaf process $\xi = \{\xi_n\}_n$. However, in order to avoid clutter, we will use a few shorthand. Under the hypothesis that the leaf process $\xi = \{\xi_n\}_n$ is an i.i.d sequence following distribution $Q$, we will write $\Pl{T,x;Q}$ and $\El{T,x;Q}$. When $Q = {\rm Ber}(p)$, we will write 
$\Pl{T,x;p}$ and $\El{T,x;p}$. And under the i.i.d. setting with distribution $Q$, we will refer to the TBRW by $Q$-TBRW, and in the Bernoulli case, $p$-TBRW. This way we avoid having to mention that the leaf process is an i.i.d. process distributed as $Q$ or ${\rm Ber}(p)$.

For the particular initial state $\{(\{o,x\},x)\}$, that is, an edge with the walker starting at the non-root tip of it, we will omit the emphasis on the initial state, writing simply $\Pl{\xi}$ and $\El{\xi}$ for a general leaf process $\xi$. For the i.i.d. case, we will simply write $\Pl{Q}$ and $\El{Q}$ for general $Q$ and $\Pl{p}$ and $\El{p}$ for the Bernoulli case.

We will conceptualize our trees as growing downwards. While this may not reflect reality, it simplifies the drawing process and allows us to utilize terminology derived from genealogical trees. This way we can refer to the {\it bottom} (or base) of the tree which corresponds to the set of the furthest leaves in the tree. See Figure~\ref{fig:tbrw} below for a reference.
\begin{figure}[h]
    \centering
    \includegraphics[width=0.3\linewidth]{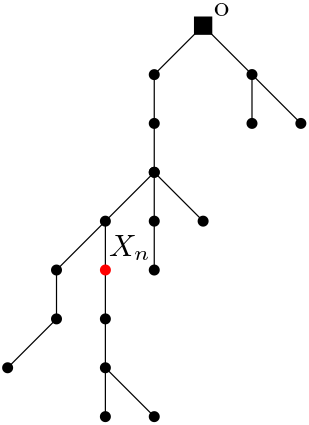}
    \caption{$(T_n,X_n)$ after $n$ steps. Walker's position is in red and the root $o$ is the square vertex.}
    \label{fig:tbrw}
\end{figure}

\subsubsection{The TBRW and Markov Property}\label{sss:markov}
Due to its dynamics, the $T$ and $X$ components of the TBRW are not Markov Chains when considered separately. However, when the sequence $\{\xi_n\}_n$ is independent, we have that the pair $\{(T_n,X_n)\}_n$ is a Markov chain on the space of rooted trees. Thus,  Markov properties hold for the process $\{(T_n,X_n)\}_n$. The remark below highlights how Markov properties work on the context of the TBRW. 
\begin{remark}[Markov Property] Consider a TBRW starting from $(T,x)$ and having an independent sequence $\xi = \{\xi_n\}_n$. For a fixed integer $m$, let $\xi^{(m)}$ be the new sequence $\{\xi_{m+n}\}_n$. Also let $\{\mathcal{F}_n\}_n$ denote the canonical filtration and $\theta_m$ the shift operator by $m$ steps. Then for all bounded, positive and measurable function $\psi$, the Simple Markov Property gives us that 
\begin{equation*}
    \El{T,x; \xi} \left[ \psi \circ \theta_m \; \middle | \; \mathcal{F}_m\right] = \El{T_m,X_m; \xi^{(m)}}[\psi],
\end{equation*}
$\Pl{T,x;\xi}$-almost surely. The important message here is that the sequence $\xi$ is shifted as well. However, if the leaf process is i.i.d., then shifting it has not effect on its distribution. Thus, in the i.i.d. setting, we will omit the shift superscript on the leaf process writing
\begin{equation*}
    \El{T,x; Q} \left[ \psi \circ \theta_m \; \middle | \; \mathcal{F}_m\right] = \El{T_m,X_m; Q}[\psi].
\end{equation*}
\end{remark}

\subsubsection{The TBRW as a Random Graph Model}\label{sss:graph} The TBRW is multifaceted and can be analyzed from several perspectives. In the present paper, we direct our attention to the $X$ (walker) component. However, a natural question arises concerning the graph structure of the $T$ (tree) component. In \cite{EIR21}, the authors explore this aspect, presenting results on the empirical degree distribution of $T_n$. Specifically, they showed that for an independent leaf process $\xi$ defined as $\xi_n = {\rm Ber}(n^{-\gamma})$, with the parameter $\gamma \in (2/3,1]$, the degree distribution exhibits a power-law behavior with an exponent of $3$. More formally, they show that for any $\gamma\in (2/3,1]$ and $d\in \mathbb{N}$
$$
\lim_{n \to \infty} \frac{\# \text{ of vertices with degree exactly }d \text{ in } T_n}{|T_n|} = \frac{4}{d(d+1)(d+2)},
$$
almost surely.

\subsubsection{The TBRW as a Random Walk in Changing Environments} \label{sss:rwce}
G. Amir, I. Benjamini, O. Gurel-Gurevich and G. Kozma introduced in \cite{Amir_Changing} the so-called Random Walk in Changing Environments (RWCE). In a nutshell, the walker walks on a fixed graph whose edges' conductances change over time according to some random rule that might depend on the history of the process and some extra randomness. 

In their notation, the pair $(X_t,G_t)$ denotes the position of the walker at time~$t$ and $G_t $ is a triple $G_t = (V,E,C_t)$, where $V$ is the vertex set, $E$ is the edge set and $C_t : E \to [0,\infty)$ gives the conductances at time $t$. The TBRW can be seen in this framework as a RWCE on an infinite tree where each vertex has a countably infinite number of neighbors.  However, except from a finite number of edges, all the edges have conductance set to zero. Then, the process of adding vertices can be seen as changing the conductances of some edges incidents to the walker's position from zero to one.

The authors then define the notion of monotone increasing (decreasing) RWCE, meaning that $C_{t+1}(e) \ge C_{t}(e)$ ( $C_{t+1}(e) \le C_t(e)$) for all $t$ and $e\in E$. Their work addresses transient RWCE on trees, but specifically under the assumption of monotone decreasing. In Theorem 5.2 in \cite{Amir_Changing}, they establish that monotone decreasing RWCEs, with an additional condition, exhibit transience. This, however, is not applicable to our context. When viewing the TBRW as an RWCE, it distinctly characterizes a monotone increasing RWCE.

\subsubsection{Theory of chain walking catalysis}\label{sss:cw} Models such as the TBRW are also applicable in chemistry. The theory of ``chain walking" in catalysis refers to a mechanism in polymerization reactions where a catalyst not only facilitates the addition of monomers to a growing polymer chain but also has the ability to walk along the chain it has created. This mechanism can result in the repositioning of the catalyst along the polymer backbone, allowing for the insertion of monomers at different points along the chain.

Note that the chain-walking mechanism can be naturally modeled by a random walk (representing the catalyst) with the capability to attach vertices (monomers) as it walks. In \cite{DS22}, the authors explore simulations of the CW (chain walk), which closely resembles the TBRW model. However, in their study, vertex degrees are at most $3$. Additionally, the walker (catalyst) flips a coin to decide whether to attach a new vertex (monomer) at its current position or to move to another position. We direct the reader to the Conclusion section in \cite{DS22}, where the authors discuss the similarities and differences between the two models in detail.

\subsubsection{The Speed and the Uniform Elliptic condition for the TBRW} Under some conditions, the TBRW is not only transient, but ballistic, which means it moves away from the root at linear speed. Ballistic behavior emerges under a sort of uniform elliptic condition over the leaf process. If $\xi$ is an leaf process, we say that the TBRW, or the leaf process, is uniformly elliptic if
\begin{equation}\tag{UE}\label{def:UE}
    \inf_{n\in \mathbb{N}}P(\xi_n \ge 1) = \kappa > 0.
\end{equation}
If $\xi = \{\xi_n\}_n$ is independent, then \eqref{def:UE} means that at each step, the TBRW has probability at least $\kappa$ of attaching at least one leaf to its position. This is similar in spirit to the uniform elliptic condition for RWRE which means that regardless the environment, the walker always has probability  at least $\kappa$ of jumping to any particular direction.

In \cite{IRVZ22} the authors showed that independence of the leaf process and \eqref{def:UE} combined lead to ballistic behavior. That is, if $\xi$ is an independent and uniformly elliptic leaf process, then 
$$
\liminf_{n \to \infty} \frac{\dist{X_n}{o}}{n} > 0, \; \Pl{T,x;\xi}\text{-almost surely,}
$$
for any initial state $(T,x)$, where $\dist{X_n}{o}$ is graph distance between $X_n$ and the root $o$.

In the i.i.d. setting, if $Q \in \pspace$ (the family of probability distributions over $\mathbb{N}$), then \eqref{def:UE} reduces to $Q(\{1,2,\dots, \}) = \kappa > 0$. The family of such distributions will be important to us, then, fixed $\kappa \in (0,1]$, we let $\mathcal{Q}_\kappa$ be the following
\begin{equation}\label{def:qkappa}
    \mathcal{Q}_\kappa := \Lbrace Q \in \pspace \; : \; Q({1,2,\dots}) \ge \kappa\Rbrace.
\end{equation}
In this work we will investigate the limit speed deeper in the i.i.d. setting for $Q \in \mathcal{Q}_\kappa$ for some $\kappa > 0$.

\subsection{Main Results: The Renewal Structure and Tail probabilities.} As mentioned earlier, the central aim of this work is to establish foundational grounds for the analysis of the TBRW. Our framework consists of a renewal structure that segments the process into i.i.d instances, together with a uniform (over the distributions of the leaf-processes) upper bound for the tail of the first renewal time. A combination of these two results permits us to derive several limit theorems concerning the walker's speed, which will be discussed in Section \ref{ss:limitthm}

Thus, let us define the sequence of renewal times. We begin setting $\renew{0} \equiv 0$ and defining $\renew{1}$ as follows
\begin{equation}\label{def:renew}
    \renew{1} = \inf \{ n> 0 : \degr{T_n}{X_n} = 1,  \dist{X_{s}}{o} < \dist{X_{n}}{o} \le \dist{X_t}{o}, \forall s<n, \forall t > n \},
\end{equation}
where $\degr{T_n}{X_n}$ denotes the degree of the walker's position at time $n$ in the tree $T_n$. In words, $\renew{1}$ is the first time $X$ reaches distance $\dist{X_{\renew{1}}}{o}$ from the root by stepping on leaf $x$ and never visits its {\it father} $f(x)$ again, where $f(x)$ is the neighbor of $x$ closer to the root $o$.
We then define $\renew{k}$ using the shift operator $\theta$
\begin{equation}\label{def:renewk}
    \renew{k} := \renew{1}\circ \theta_{\renew{k-1}} + \renew{k-1}.
\end{equation}
Notice that Theorems \ref{t:gtailrenew} and \ref{t:renew} below guarantee that $\renew{k}$ is finite for all $k$. Our first main result concerns the tail of $\renew{1}$.
\begin{theorem}[Uniform Tail bounds for $\renew{1}$]\label{t:gtailrenew} Fix $\kappa \in (0,1]$. Then, there exist positive constants $C$ and $C'$ depending on $\kappa$ only such that 
 \begin{equation*}
    \sup_{Q\in\mathcal{Q}_k}\Pl{Q} \left( \renew{1} > t \right) \le Ce^{-C't^{1/2}}, 
 \end{equation*}
  for all $t$.
\end{theorem}
Our second main result concerns the renewal structure itself. In the forthcoming statement and throughout this paper, we will need a notation for the hitting time to a specific vertex. Thus, given a vertex $x$, we let $H_x$ be 
\begin{equation}\label{def:hittime}
    H_x := \inf \{ n \ge 0 \; : \; X_n = x\}.
\end{equation}
With all the above notation in mind, we can state properly our renewal structure.
\begin{theorem}[The Renewal Structure]\label{t:renew} Let $Q$ be a probability distribution in $\mathcal{Q}_\kappa$ for some positive $\kappa$, and  $(T,x)$ a finite initial state. Then, the following sequence of random vectors $\{\left( \dist{X_{\renew{k}}}{o} - \dist{X_{\renew{k-1}}}{o},\renew{k}- \renew{k-1}\right)\}_{k\ge 1}$ is a sequence of independent random vectors. Moreover, for $k>1$,the random variables $\renew{k}-\renew{k-1}$ and $\dist{X_{\renew{k}}}{o} - \dist{X_{\renew{k-1}}}{o}$ are respectively distributed as~$\renew{1}$ and $\dist{X_{\renew{1}}}{o}$ conditioned on the event $\{H_o = \infty\}$ and started from $(\{o,x\},x)$. That is, for $k>1$,
$$
P_{T,x; Q}\left(\renew{k} - \renew{k-1} \in \cdot \right) = P_{Q}\left(\renew{1} \in \cdot  \, \middle | \, H_o = \infty \right)
$$
and
$$
P_{T,x; Q}\left(\dist{X_{\renew{k}}}{o} - \dist{X_{\renew{k-1}}}{o} \in \cdot \right) = P_{Q}\left(\dist{X_{\renew{1}}}{o} \in \cdot  \, \middle | \, H_o = \infty \right).
$$
\end{theorem}
In words, between each epoch $[\renew{k}, \renew{k+1}]$, a $Q$-TBRW is a $Q$-TBRW starting from an edge and conditioned to never visit the root $o$.

\subsection{Main Results: Limit Theorems}\label{ss:limitthm}
A consequence of our renewal structure is that different limit theorems for the distance hold true in the i.i.d setting with $Q \in \mathcal{Q}_\kappa$. One of them being a SLLN, which means a well-defined constant speed for the TBRW.

For the Bernoulli case,  in \cite{FIORR21} the authors prove a SLLN for the distance. More precisely, they show that for any finite initial state $(T,x)$, the walker has a well-defined speed, meaning that there exists a constant $v = v(p)$ depending on $p$ only, such that
$$
\lim_{n \to \infty} \frac{\dist{X_n}{o}}{n} = v(p), \; \Pl{T,x;p}\text{-almost surely}
$$
Thus, in this setting, the walker is {\it transient} and {\it ballistic} with a well-defined speed. In this paper we generalize their result to general distributions.
\begin{theorem}[Strong Law of Large Numbers]\label{t:lln} Consider $Q$-TBRW starting from a finite initial state~$(T,x)$. Then, 
$$
\lim_{n \to \infty} \frac{\dist{X_n}{o}}{n} = \frac{E_{ Q}\left[ \; \dist{X_{\renew{1}}}{o} \; \middle | \;  H_o = \infty \right]}{E_{ Q}\left[ \; \renew{1} \; \middle | \; H_o = \infty \right]} =: v(Q), \; P_{T,x; Q}\text{-a.s.}
$$
Moreover, $v(Q) = 0$ if, and only if, $Q = \delta_0$, that is, $Q(\{0\})=1$.
\end{theorem}
Under the RWCE framework, the above result gives a class of monotone increase RWCE which are not only transient but ballistic and have a well-defined deterministic speed. This case is not covered by G. Amir, I. Benjamini, O. Gurel-Gurevich, and G. Kozma in \cite{Amir_Changing}.

In Section \ref{s:noniid} we discuss the case of non-identically distributed leaf processes. We also construct an example of a TBRW with an independent leaf process $\{\xi_n\}_n$ for which the limit of $\dist{X_n}{o}/n$ does not exist, that is, the TBRW does not have a well-defined speed. The example illustrates that if one wants to drop the identically distributed hypothesis from Theorem \ref{t:lln}, some sort of convergence or mixing condition on the leaf process $\{\xi_n\}_n$ must be imposed.

The second type of limit theorem we have is a Law of the Iterated Logarithm. Given $Q \in \mathcal{Q}_\kappa$ for some $\kappa$, we let $\sigma^2(Q)$ be the following
\begin{equation}\label{def:var}
    \begin{split}
    \sigma^2(Q)  &  = \El{Q}\left [ ( \dist{X_{\renew{1}}}{o} - \El{Q}\left [ \dist{X_{\renew{1}}}{o}  \middle | H_o = \infty \right])^2 \middle | H_o = \infty \right].
    \end{split}
\end{equation}
That is, $\sigma^2(Q)$ is the variance of $\dist{X_{\renew{2}}}{o} - \dist{X_{\renew{1}}}{o}$.
\begin{theorem}[Law of the Iterated Logarithm]\label{t:lil} Consider a $Q$-TBRW starting from a finite initial state~$(T,x)$. Then, 
$$
\limsup_{n \to \infty}\frac{|\dist{X_n}{o} - n\cdot v(Q)|}{\sigma(Q)\sqrt{2n\log \log n}} = 1, \; P_{T,x; Q}\text{-a.s.}
$$
\end{theorem}
Finally, we also establish CLT-type theorems. That is, a CLT and an Invariance Principle (Functional CLT) for the distance of the TBRW.
\begin{theorem}[Central Limit Theorems]\label{t:clts} Consider a $Q$-TBRW starting from a finite initial state~$(T,x)$. Then, 
\begin{enumerate}
    \item[(a)] $$
\frac{\dist{X_{n}}{o} - n\cdot v(Q)}{\sigma(Q)\sqrt{n}} 
$$
converges in law to a standard normal distribution.
    \item[(b)] $$
\frac{\dist{X_{\lfloor nt \rfloor}}{o} - {\lfloor nt \rfloor}v(Q)}{\sigma(Q)\sqrt{n}} 
$$
converges in law to $B(t)$ where $\{B(t)\}_{t\ge 0}$ is a standard Brownian Motion.
\end{enumerate}

\end{theorem}
\subsection{Main Results: Continuity of the Speed.} Theorem \ref{t:lln} ensures that $Q$-TBRW have a well-defined speed $v$ depending only on $Q \in \pspace$. A consequence of this result is that $v$ defines as function from $\pspace$ to $\mathbb{R}$. In the second part of this work, we explore this perspective by considering $\pspace$ with the total variation distance
\begin{equation}
    \dtv{Q}{Q'} := \frac{1}{2}\sum_{i=0}^{\infty}|Q(i) - Q'(j)|.
\end{equation}
This way, $\pspace$ becomes a metric space, which allows us to investigate analytical properties of the speed $v$. Our next contribution states that the speed is a continuous function.
\begin{theorem}[The Speed is continuous]\label{t:1}  The speed function $v: \pspace \to \mathbb{R}$ given by Theorem \ref{t:lln}
is continuous. More precisely, if $\{Q_j\}_j$ is a sequence of distributions on $\pspace$ and $Q \in \pspace$, then
$$
\lim_{j \to \infty}\dtv{Q_j}{Q} = 0 \implies \lim_{j\to \infty}v(Q_j) = v(Q).
$$
\end{theorem}
In personal communication, it was suggested to the authors, by professor Y. Peres, the investigation of the speed $v$ in the Bernoulli case. That is, $v$ as a function of the parameter $p$. The above theorem gives analytical information about $v(p)$: it is a continuous function on $[0,1]$, consequently uniformly continuous. At Section \ref{s:finalc} we discuss other problems involving $v(p)$.

\subsection{Technical Challenges and Proof Ideas}
We finish this introduction with a high-level discussion about challenges and some proof ideas. \\

\noindent \underline{\it Chicken-Egg problem.} The fact that the geometry of the environment cannot be decoupled of the walker's trajectory leads to a chicken-egg situation in the analysis of the TBRW. For instance, the degree distribution of the environment has a direct impact on the walker's speed. However, the degree distribution depends on the walker's trajectory, since the walker might need to visit a vertex many times to increase its degree. Thus, in the analysis of the TBRW, we often face questions of the form: to understand the position of the walker, one needs to investigate the geometry of its environment. However, to investigate properties of its environment, one needs to understand the walker's trajectory. \\

\noindent \underline{\it Limit Theorems.} The limit theorems we have are consequence of the renewal structure together with a first renewal time with a extremely light tail.  This light tail allows us to leverage the framework constructed by P. Glynn and W. Whitt in \cite{GW93}. We highlight that, in general, proving results regarding renewal times can be a hard problem. Quoting Sznitman and Zerner in \cite{SZ99} {\it``The renewal time $\renew{1}$ is rather complicated, and the extraction of information on its tail is not straightforward."}

To control the tail of $\renew{1}$, we investigate a sequence of stopping times that are natural candidates for regeneration times. These stopping times are related to those times in which the walker hits the very bottom of the tree. Since they are hitting times, we can decompose the tail event of $\renew{1}$ as a union of events involving hitting times only. This way we leverage Strong Markov property.\\

\noindent \underline{\it Continuity of the speed.} In a nutshell, the two main challenges involving proving analytical properties of the speed for the Bernoulli case are: (1) {\it lack of any expression for $v$;} and (2) {\it the absence of a monotone coupling to compare different instances of the model.}

Regarding (1), in \cite{FIORR21} the authors show the SLLN for the Bernoulli case without providing any type of expression for $v(p)$. In a high-level, their work is similar to the work of Lyons, Pemantle and Peres \cite{lyons1995} on random walk on Galton-Watson trees. However, in that paper the stationary measure can be described explicitly (leading to an explicit formula for the speed), which is not the case in the context of the TBRW. Thus, in order to overcome (1), we first leverage the renewal structure, which gives us an expression for $v$ involving the first moment of $\renew{1}$. 

Once we have overcome (1), we are then left with the task of essentially showing that the first moment of $\renew{1}$ is continuous on $Q$. And for that, we need two ingredients: one is some sort of uniformity on $Q$ for the upper bound for the tail of $\renew{1}$ and a coupling to compare different instances of the model for different distributions $Q$'s. In order to reach that goal, we introduce a coupling inspired by the monotone coupling from bond percolation. However, in our case with probability one eventually the walkers under different distributions split up. With the uniformity on the tail of $\renew{1}$ and the coupling, we are able to show that for $Q$ and $Q'$ that are close enough of each other, in total variation distance, the walkers are more likely to regenerate before splitting up. This can be seen as the continuity of the first moment of $\renew{1}$.

\section{Tail bounds for $\renew{1}$: Proof of Theorem \ref{t:gtailrenew}}
Since our results are consequence of the renewal structure provided by Theorem~\ref{t:renew} together with the tail bounds given by Theorem \ref{t:gtailrenew}, we begin showing all the results related to the renewal structure first. Usually, extraction of distributional information about renewal times tend to be hard. Thus, our approach will be replacing $\renew{1}$ by a sequence of stopping times, which will allow us to leverage Strong Markov property in its analysis.

Let us introduce these auxiliary stopping times. We set  $\lstop_0 \equiv 0 $ and define
\begin{equation}\label{def:sigma1}
    \lstop_{1} := \inf \{ n> \lstop_{0} \, : \,  \degr{T_n}{X_n} = 1, \; \dist{X_n}{o} >\dist{X_{\lstop_{0}}}{o} \},
\end{equation}
with the convention that the {\it infimum} over the empty set is infinity. And in general, the $k$th stopping time is defined as $\lstop_k =  \lstop_1 \circ \theta_{\lstop_{k-1}} + \lstop_{k-1}$, on the event $\{\lstop_{k-1} < \infty \}$ and $\lstop_k = \infty$ on its the complement. We highlight that under the i.i.d. setting, $\lstop_k$ is finite $\Pl{T,x;Q}$-almost surely for any finite initial state $(T,x)$ and $Q\in \mathcal{Q}_\kappa$. This is a consequence of Theorem 1.3 in \cite{IRVZ22}, which guarantees transience.

We will also need to keep track on the time needed to hit the father of a leaf. For this reason we introduce for $k\in \mathbb{N}$, the following stopping times
\begin{equation}\label{def:rtime}
    \rtime_{k} :=   H_{f(X_0)}\circ \theta_{\eta_k} + \eta_k
\end{equation}
That is, $\rtime_k$ is the first visit to the father of $X_{\eta_k}$ after time $\eta_k$. For the proof of Theorem \ref{t:gtailrenew} we will need the three auxiliary results below.
\begin{lemma}\label{l:return} Fix $\kappa \in (0,1]$ and let $(T,x)$ be a finite initial state and $y \in T$. Then,
$$
 \sup_{Q\in \mathcal{Q}_\kappa}\Pl{T,x;Q}(H_{y} < \infty ) = \Pl{T,x;\kappa}(H_{y} < \infty ).
$$
\end{lemma}
\begin{lemma}\label{l:rtime2}Let $Q \in \mathcal{Q}_\kappa$ for some $\kappa >0$. Then, for all~$k$
    \begin{equation}\label{ineq:Hbound}
        \Pl{Q}\left( \max_{j \le k} \rtime_j < \infty \right) \le (k+1) \left[\Pl{Q}\left(  H_{o} < \infty \right)\right]^{k+1}.
    \end{equation}
\end{lemma}
\begin{lemma}\label{l:taildeltasigma} Fix $\kappa \in (0,1]$ and let $\mathcal{T}_*$ be the set of pairs $(T,x)$ where $T$ is a finite rooted tree and $x$ is a leaf at maximal distance from the root. Then, there exists  a positive constant $C$ depending on $\kappa$ only such that
\begin{equation}
       \sup_{Q\in \mathcal{Q}_\kappa}\sup_{(T,x)\in \mathcal{T}_*}\Pl{T,x;Q}\left( \lstop_k - \lstop_{k-1} > t \right) \le e^{-Ct},
    \end{equation}
    for all $t \in \mathbb{N}$. 

\end{lemma}
We will defer their proofs to the next two subsections. For now, we will focus on showing how to obtain Theorem \ref{t:gtailrenew} from them.
\begin{proof}[Proof of Theorem \ref{t:gtailrenew}: Tail bounds for $\renew{1}$] It is important to recall that our initial state throughout this proof is a rooted edge $\{o,x\}$ with the walker starting at~$x$. Also, we will write $C$ and $C'$ for constants and they may change from line to line. Thus, for any~$t \in \mathbb{N}$, the following identity of events holds
\begin{equation}\label{eq:tailevent}
    \Lbrace \renew{1} > t\Rbrace = \bigcup_{k=1}^t\Lbrace \lstop_{k-1} \le t, \max_{j \le k-1} \rtime_{j} < \infty, \lstop_k > t\Rbrace.
\end{equation}
That is, $\renew{1}$ is larger than $t$ if either the walker visits the root and does not see itself at a leaf before time $t$ or it sees itself at a leaf, but eventually it visits the father of that leaf.


Now, observe that, by Lemma \ref{l:taildeltasigma} and union bound, it follows that for any $k$, 
\begin{equation}\label{ineq:sigmaktail}
    \Pl{Q}\left( \lstop_k > t \right) \le \sum_{j=1}^k  \Pl{Q}\left( \lstop_j - \lstop_{j-1} > t/k \right) \le ke^{-Ct/k},
\end{equation}
in which, the constant $C$ depends only on $\kappa$. Moreover, by Theorem 1.3 in \cite{IRVZ22}, under our settings, the walker is transient, which implies that $\beta:=\Pl{\kappa}(H_o < \infty) < 1$. Consequently, by Lemmas \ref{l:return} and \ref{l:rtime2}
\begin{equation}\label{ineq:bhtilde}
    \Pl{Q}\left( \max_{j \le k} \rtime_j < \infty \right) \le (k+1)[\Pl{Q}\left( H_o < \infty \right)]^{k+1} \le (k+1)e^{-\beta(k+1)},
\end{equation}
for any $Q \in \mathcal{Q}_\kappa$. Finally, using \eqref{eq:tailevent} and \eqref{ineq:bhtilde} and \eqref{ineq:sigmaktail} gives us
\begin{equation}
    \begin{split}
        \Pl{Q} (\renew{1} > t) & \le \sum_{k=1}^{t^{1/2}} \Pl{Q}(\lstop_k > t) + \sum_{k=t^{1/2}}^{t} \Pl{Q}\left(\max_{j \le k-1} \rtime_j < \infty\right) \\
        & \le \sum_{k=1}^{t^{1/2}} ke^{-Ct/k} + \sum_{k=t^{1/2}}^{t} (k+1)e^{-C'(k+1)} \\
        & \le C'te^{-Ct^{1/2}} + Ct^2e^{-C't^{1/2}}.
    \end{split}
\end{equation}
By adjusting the constants one obtain the desired bound.
\end{proof}

\subsection{Proof of Lemmas \ref{l:return} and \ref{l:rtime2}}
We start by Lemma \ref{l:return}. Given $Q \in \mathcal{Q}_\kappa$, we will construct a coupling between a $Q$-TBRW and a $\kappa$-TBRW both starting from $(T,x)$. The coupling will be constructed in a way that if the $Q$-TBRW has visited $y \in T$, then the $\kappa$-TBRW has also visited it. 

\begin{proof}[Proof of Lemma \ref{l:return}] Before we start the proof, we refer the reader to Figure \ref{fig:f_coupling_pq} below which might be helpful for a better understanding of the coupling to be constructed. 
    \begin{figure}[h]
        \centering
        \includegraphics[width=0.4\linewidth]{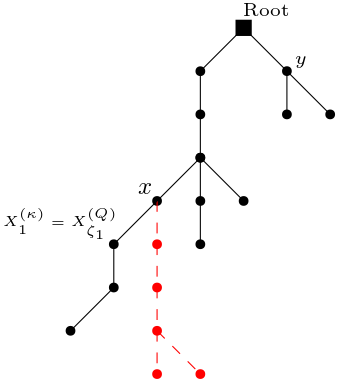}
        \caption{The red structure represents an excursion of the $Q$-TBRW not visible to the $\kappa$-TBRW.}
        \label{fig:f_coupling_pq}
    \end{figure}

In our coupling the two instances of the TBRW walk on the same tree, however, some of the vertices might be visible only to the $Q$-TBRW and others only to the $\kappa$-TBRW. We denote by $\mathbf{T}$ the tree over which both walkers are walking on. The vertices of $\mathbf{T}$ have one of the three following labels: $\kappa$, $q$ or $q\kappa$. Vertices with the third label are visible to both walkers, whereas vertices labeled $\kappa$ (resp. $q$) are visible to the $\kappa$-TBRW (resp. $Q$-TBRW) only.

    
We will use an i.i.d. sequence $\{U_n\}_n$, where $U_1 \sim {\rm Uni}[0,1]$, as the common source of randomness to sample $\{\xi_n^{(Q)}\}_n$ and $\{\xi_n^{(\kappa)}\}_n$. We will generate $\xi_n^{(Q)}$ in the following way:
\begin{enumerate}
    \item  Conditioned on $\{U_n \le \kappa\}$, we put $\xi_n^{(Q)} = j$, with $j$ sampled independently from the past from the probability distribution $\mu$ over $\{1,2,\dots,\}$ defined as follows
    $$
    \mu(j) := \frac{Q(j)}{Q(\{1,2,\dots\})}.
    $$
    \item Whereas, conditioned on $\{U_n > \kappa\}$, we put $\xi_n^{(Q)} = j$ with $j$ being sampled independently from the past from the probability distribution $\nu$ over $\{0,1,2,\dots\}$ defined as
    $$
        \nu(j) := \frac{Q(0)}{1-\kappa}\delta_0(j) + \frac{Q(j)}{1-\kappa}\left( 1 -\frac{\kappa}{Q(\{1,2,\dots \})}\right)(1-\delta_0(j)),
    $$
\end{enumerate}
where $\delta_0$ is the probability distribution such that $\delta_0(\{0\}) = 1$. Notice that  the probability distribution $\nu$ is well-defined since $Q(\{1,2,\dots,\}) \ge \kappa$. Moreover, from the above definition, we have that $\{\xi^{(Q)}_n\}_n$ is an i.i.d sequence with $\xi_1^{(Q)}\sim Q$. 

We want to couple the two instances of the TBRW starting from the same initial state $(T,x)$. For this reason, we start setting $\mathbf{T}_0$ as $T$ with all of its vertices labeled as $q\kappa$. We also set $X_0^{(Q)} = X_0^{(\kappa)} = x$. Then, we obtain $\mathbf{T}_1$ as follows. We attach~$\xi_1^{(Q)}$ vertices to $x$. If $U_1 \le \kappa$, we label one of the $\xi_1^{(Q)}$ vertices $q\kappa$, and the (possible)~$\xi_1^{(Q)} -1$ remaining vertices receive label $q$. If $U_1 >\kappa$, then we label them, if any, $q$.

Then, we let $X_1^{(Q)}$ to be a neighbor of $X_0^{(Q)}$ that is uniformly chosen with label~$q$ or $q\kappa$. If the $Q$-TBRW jumps to a vertex labeled $q$, we need to wait for it to return to a vertex that is visible to the $\kappa$-TBRW. For this reason, we also need to keep track on times when the $Q$-TBRW moves on the common structure, that is, when it jumps only over $q\kappa$-labeled vertices. We set $\zeta_0 \equiv 0$ and define 
\begin{equation}
    \zeta_1 := \inf\{ n > \zeta_0 \; : \; {\rm label }(X^{(Q)}_{n-1}) = {\rm label }(X^{(Q)}_{n}) = q\kappa\}.
\end{equation}
In words, $\zeta_1$ is the first time the $Q$-TBRW moves over portions of $\mathbf{T}$ seen by both walkers. 

If $\zeta_1 < \infty$, then we let $X_1^{(\kappa)} = X_{\zeta_1}^{(Q)}$. We also label all vertices added in the time interval $(0,\zeta_1)$ as $q$. Due to the tree structure of the graph and that $X^{(Q)}_0$ has label $q\kappa$, $X^{(Q)}_{\zeta_1}$ is a neighbor of $X^{(Q)}_{0}$ with label $q\kappa$. And conditioned on the $Q$-TBRW having walked on the common structure seen by both walkers, $X_1^{(\kappa)}$ is distributed as a single step of a simple random walk on $\mathbf{T}_{\zeta_1}$ with all the $q$-vertices removed. On the event $\zeta_1 = \infty$ we move the $\kappa$-TBRW process independently using a fresh source of randomness for it.

In general, we define $\zeta_k$ as the $k$th jump of the $Q$-TBRW on the portions of $\mathbf{T}$ seen by both walkers. If $\zeta_k$ is finite for $k>1$, then $X_{\zeta_k}^{(Q)} =  X_k^{(\kappa)}$, and we can repeat the procedure above described. In terms of trees, $T^{(\kappa)}_k$ is obtained from $\mathbf{T}_k$ by removing all the $q$-vertices. Whereas $T^{(Q)}_k$ is obtained by removing the $\kappa$-vertices.

Notice that in this coupling, if the $Q$-TBRW visited $y \in T$, where $T$ is the initial tree, it must have visited it at some time $\zeta_k$, due to the tree structure. This means, by construction, that the $\kappa$-TBRW also visited $y$. In symbols, the following inclusion of events holds
$$
\{H_{y}^{(Q)} < \infty \} \subseteq \{H^{(\kappa)}_{y} < \infty \},
$$
which gives us the result.
\end{proof}
We point out that in the above proof, if $Q(\{2,3,\dots,\}) > 0$ or $Q(\{1\}) > \kappa$, then there is a positive probability that the $Q$-TBRW adds a vertex that is not visible to the $\kappa$-TBRW in the first step. And since both processes are transient, there is also a positive probability that $\zeta_1 = \infty$ and then the $Q$-TBRW does not visit $y$.  In that case, the $\kappa$-TBRW starts evolving independently and there is a positive probability of visiting $y$. In this case, 
$$
\Pl{T,x;Q}(H_y^{(Q)} < \infty) < \Pl{T,x;\kappa}(H_y^{(Q)} < \infty).
$$.

In order to show Lemma \ref{l:rtime2}, we will need the following auxiliary result.
\begin{lemma}\label{l:rtime3}  Let $\mathcal{T}$ be the set of pairs $(T,x)$, where $T$ is a finite rooted tree and $x$ is a vertex of $T$ different than the root, and fix $Q\in \pspace$. Then, 
    $$
        \sup_{(T,x) \in \mathcal{T}}\Pl{T,x;Q}(H_{f(x)} < \infty) = \Pl{Q}(H_o < \infty).
    $$
In words, having a structure below $x$ makes harder to visit its father $f(x)$.
\end{lemma}
\begin{proof} The proof follows from a coupling similar to the one we constructed in the proof of Lemma \ref{l:return}. In this case, we have two instances of a $Q$-TBRW, $\{(T'_n,X'_n)\}_n$ and $\{(T_n,X_n)\}_n$. The former starts from $(T,x)$, whereas the latter starts from $(\{o,x\},x)$. 

Again, the process $\{(T_n,X_n)\}_n$ waits for $\{(T'_n,X'_n)\}_n$ whenever the latter moves to a vertex which is not available to the former. If the prime TBRW does not come back to a structure that is common to both processes, we let the former process move independently. By this construction, it follows that the process starting from the edge is more likely to visit $o$ than the prime process to visit $f(x)$.
    
\end{proof}
Now we are ready to prove Lemma \ref{l:rtime2}
\begin{proof}[Proof of Lemma \ref{l:rtime2}]
     We begin recalling the definition of $\rtime_{k}$.
     $$
     \rtime_{k} :=   H_{f(X_0)}\circ \theta_{\eta_k} + \eta_k.
     $$
     Clearly, $\rtime_{k} > \eta_k$. Also notice that the following bound holds
\begin{equation}\label{ineq:H1}
    \begin{split}
        \Pl{Q}\left( \max_{j \le k} \rtime_j < \infty \right) & \le \Pl{Q}\left( \max_{j \le k-1}\rtime_j < \lstop_{k}, \rtime_k < \infty \right) + \Pl{Q}\left( \lstop_k < \max_{j \le k} \rtime_j < \infty \right).
    \end{split}
\end{equation}
We will bound the second term on the RHS of the above inequality first. Recall that we are starting from an edge with the walker at the non-root tip of it. For this reason and the definition of $\lstop_1$ given at \eqref{def:sigma1} at page \pageref{def:sigma1}, on the event 
$$
\{\lstop_k < \max_{j \le k} \rtime_j < \infty\}
$$ 
we have that $X_{\lstop_k}$ is at distance $k+1$ from $o$ and $f(X_{\lstop_j}) = X_{\lstop_{j-1}}$, for $j\le k$. Indeed, since in this event $\lstop_k$ occurs before that the walker visits $f(X_{\lstop_j})$ for $j\le k$, at time $\lstop_j$, for $j<k$ the walker is at leaf and it needs to add at least one leaf to its position and jump to one of them, otherwise, $\rtime_j$ occurs before $\lstop_k$. Thus, $\lstop_{j+1} = \lstop_j+1$ and $f(X_{\lstop_{j+1}}) = \lstop_{j}$. Now, given a tree of height at leat $k+1$ with the walker at distance $k+1$ from the root, denote by $x_j$ the vertex at distance $j$ from the root in the path connecting the walker to the root. With this notation in mind we have that 
\begin{equation}
    \mathds{1}\{\lstop_k < \max_{j \le k} \rtime_j <\infty\} =  \mathds{1}\{ \max_{j \le k} H_{x_j} <\infty\} \circ \theta_{\lstop_k}\cdot \mathds{1}\{\lstop_k < \max_{j \le k} \rtime_j\},
\end{equation}
$\Pl{Q}$-almost surely. Therefore, by Strong Markov Property (shifting the process by $\lstop_k$ first, then by $H_{x_j}$'s), using that $\rtime_j$'s are stopping time, and Lemma \ref{l:rtime3}
\begin{equation}\label{ineq:HT2}
    \begin{split}
        \Pl{Q}\left( \lstop_k < \max_{j \le k} \rtime_j < \infty \right) & = \El{Q}[\mathds{1}\{\lstop_k < \max_{j\le k-1} \rtime_j \}\Pl{T_{\lstop_k}, X_{\lstop_k}; Q}(\max_{j \le k} H_{x_j} <\infty)] \\
        & \le \left[\Pl{Q}(H_o < \infty)\right]^{k+1}.
    \end{split}
\end{equation}
We also have used the fact that the leaf process $\xi$ is i.i.d. in each application of the Strong Markov property.


As for the first term in the RHS of \eqref{ineq:H1}, the Strong Markov property, shifting the process by $\lstop_k$, recalling that $\lstop_k$ is finite almost surely, and using Lemma \ref{l:rtime3}, leads to
\begin{equation}\label{ineq:HT1}
    \begin{split}
         \Pl{Q}\left( \max_{j \le k-1} \rtime_j < \lstop_{k}, \rtime_k < \infty \right) & = \El{Q}[\mathds{1}\{\max_{j\le k-1} \rtime_j < \lstop_k \}\Pl{T_{\lstop_k}, X_{\lstop_k}; Q}(H_{f(X_0)} < \infty) ] \\
         &\stackrel{\text{Lemma} \ref{l:rtime3}}{\le } \Pl{Q}(H_o < \infty) \Pl{Q}\left( \max_{j \le k-1} \rtime_j < \infty \right).
    \end{split}
\end{equation}
Combining \eqref{ineq:H1}, \eqref{ineq:HT1} and \eqref{ineq:HT2}, we obtain the following recursion
\begin{equation*}
    \Pl{Q}\left( \max_{j \le k} \rtime_j < \infty \right) \le \Pl{Q}\left( \max_{j \le k-1} \rtime_j < \infty \right)\Pl{Q}(H_o < \infty) + \left[\Pl{Q}(H_o < \infty)\right]^{k+1},
\end{equation*}
which iterated $k$ times gives us the result.

\end{proof}

\subsection{Proof of Lemma \ref{l:taildeltasigma}: Tail bounds for $\lstop_k-\lstop_{k-1}$} In \cite{FIORR21}, for each $r\in \mathbb{N}$, the authors construct a sequence of stopping times $\{\sigma^{(r)}_j\}_j$ for the Bernoulli instance of the TBRW such that
\begin{enumerate}
    \item For all $j$, $1\le |\sigma_j^{(r)} - \sigma^{(r)}_{j-1}| \le e^{\sqrt{r}}$ almost surely;

    \item For all $j$, $|\dist{X_{\sigma^{(r)}_{j+1}}}{o} - \dist{X_{\sigma^{(r)}_j}}{o}| \le r$ almost surely.
  
\end{enumerate}
Then, their Lemma 5.1, stated below, shows that $r$ can be chosen large enough so that the speed of the TBRW is bounded from below by the speed of a right-biased simple random on $\mathbb{Z}$. More formally, they show the following result.
\begin{lemma*}[Lemma 5.1 of \cite{FIORR21} ]\label{lema:couplingsk} Let $T_0$ be a rooted locally finite tree and $x_0$ one of its vertices. For any $q\in (1/2,1)$
 there exists $r = r(p,q)$ depending on $p$ and $q$ only, such that the process $\{ \mathrm{dist}(X_{\sigma^{(r)}_k}, o)/r\}_{k \ge 0}$  can be coupled with a simple random walk $\left \lbrace S_k \right \rbrace_{k \ge 0}$  on $\mathbb{Z}$, whose probability of jumping to the right is equal to $q$,  in such way that
\[
\mathbb{P} \left( \mathrm{dist}(X_{\sigma^{(r)}_k}, o) \ge rS_k, \forall k \right) = 1,
\]
when the process $\left \lbrace S_k \right \rbrace_{k \ge 0}$ starts from $\lfloor \mathrm{dist}_{T_{0}}(x_0, o)/r \rfloor$. 
\end{lemma*}
For our purposes, we will need a version of the above lemma in which the constant~$r$ can be chosen uniformly across a whole family of TBRW. More precisely, we will need the following modified version
\begin{lemma}\label{l:unicoupling} Given $Q \in \mathcal{Q}_\kappa$, for some $\kappa \in (0,1]$, let $(T,x)$ be a finite initial state and $\{(T^Q_n,X_n^Q)\}_n$ be a $Q$-TBRW. Then, for any $q\in (1/2,1)$
 there exists $r$ depending on $\kappa$ and $q$ only, such that the process $\{ \mathrm{dist}(X^Q_{\sigma^{(r)}_k}, o)/r\}_{k \ge 0}$  can be coupled with a simple random walk $\left \lbrace S_k \right \rbrace_{k \ge 0}$  on $\mathbb{Z}$, whose probability of jumping to the right is equal to $q$,  in such way that
\[
\mathbb{P} \left( \mathrm{dist}(X^Q_{\sigma^{(r)}_k}, o) \ge rS_k, \forall k \right) = 1,
\]
when the process $\left \lbrace S_k \right \rbrace_{k \ge 0}$ starts from $\lfloor \mathrm{dist}(x_0, o)/r \rfloor$. 
\end{lemma}
In words, the above lemma is telling us that if we have a family of TBRW such that the probability of adding at least one leaf at each step is at least $\kappa>0$, then there is a uniform lower bound for the speed of such family of TBRW's. We will postpone its proof to the end of this section.
\begin{proof}[Proof of Lemma \ref{l:taildeltasigma}] Recall that we are starting our process from $(\{o,x\},x)$. We then begin by claiming the following.
    \begin{claim} If $X_{\eta_k}$ is a vertex at the bottom of $T_{\eta_k}$, then $X_{\eta_{k+1}}$ is vertex at the bottom of~$T_{\eta_{k+1}}$. 
    \end{claim}
    Recall our trees are finite and grow downwards. Thus, the bottom of the tree is the set of leaves at maximum distance from the root $o$.
    \begin{claimproof}If $X_{\eta_k}$ is a the bottom of $T_{\eta_k}$, thus $X_{\eta_k}$ is at a leaf of $T_{\eta_k}$ and there are no vertices below any vertex at the level of $X_{\eta_k}$. Now, for $\eta_{k+1}$ to occur, $X$ needs to reach the level below $X_{\eta_k}$ in $T_{\eta_k}$, which is empty. This means, $X$ needs to add a leaf at some vertex in the level of $X_{\eta_k}$ and jump to it. Since the walker needs to visit a vertex to increase its degree, there are no vertices in the level below $X_{\eta_{k+1}}$, which proves the claim.
    \end{claimproof}
     With the above claim in mind, by Strong Markov property we have that 
\begin{equation}
    P_{T,x;Q}( \eta_k - \eta_{k-1}> t) = E_{T,x;Q}\left[ P_{T_{\eta_{k-1}},X_{\eta_{k-1}};Q}(\eta_1 > t)  \right].
\end{equation}
 Now, notice that if we start the process from an initial condition $(T,x)$, where $x$ is a vertex at the bottom of $T$, then, by induction and the claim,  $X_{\eta_j}$ will be at the bottom of $T_{\eta_j}$ for all $j$. Since we are starting the process from $(T,x) \in \mathcal{T}_*$, at time $\eta_{k-1}$, the walker is at the bottom of $T_{\eta_{k-1}}$. This leads to the following bound almost surely
\begin{equation*}
        P_{T_{\eta_{k-1}},X_{\eta_{k-1}};Q}\left( \eta_1 > t \right) \le P_{T_{\eta_{k-1}},X_{\eta_{k-1}};Q}(\dist{X_n}{o} \le \dist{X_0}{o}, \;  \forall n \le t).
\end{equation*}
Now, setting $q=2/3$ in Lemma \ref{l:unicoupling}, we have that for any $Q\in \mathcal{Q}_\kappa$, there exists $r=r(\kappa,q)$ such that 
$\mathrm{dist}(X_{\sigma^{(r)}_k}, o) \ge rS_k$. Moreover, using that 
$$
1\le |\sigma_j^{(r)} - \sigma^{(r)}_{j-1}| \le e^{\sqrt{r}},
$$ almost surely and Hoeffding's inequality, it follows that there exists a positive constant~$C$ depending on $r$ only such that
\begin{equation*}
    \begin{split}
        P_{T_{\eta_{k-1}},X_{\eta_{k-1}};Q}\left( \eta_1 > t \right) &\le P_{T_{\eta_{k-1}},X_{\eta_{k-1}};Q}(\dist{X_n}{o} \le \dist{X_0}{o}, \;  \forall n \le t) \\
        & \le P_{T_{\eta_{k-1}},X_{\eta_{k-1}};Q}\left( S_{\lfloor t / e^{\sqrt{r}}\rfloor} \le \left \lfloor \frac{\dist{X_0}{o}}{r}\right \rfloor  \right) \\
        & \le e^{-Ct},
    \end{split}
\end{equation*}
which proves the lemma since $r$ depends on $\kappa$ only.
\end{proof}

\subsection{Proof of Lemma \ref{l:unicoupling}}
The proof of Lemma \ref{l:unicoupling} follows by an inspection on the proof of Lemma 5.1 in \cite{FIORR21}. We would like to show that Lemma 5.1 in \cite{FIORR21} also holds for a $Q$-TBRW with $Q \in \mathcal{Q}_\kappa$. We will point out bellow that the proof of Lemma 5.1 depends only on a lower bound for the probability of adding at least one vertex. Consequently, its proof can carried over to any  $Q$-TBRW, with $Q \in \mathcal{Q}_\kappa$, with same constants.
\begin{proof}[Proof of Lemma \ref{l:unicoupling}]The proof of Lemma 5.1 in \cite{FIORR21} is based on two main ideas:
\begin{enumerate}
    \item In $n$ steps, the walker reaches a nontrivial distance from the root. The nontrivial distance in this case is some large power of $\log n$. See Lemma 3.1 in \cite{FIORR21};
    \item Once it is far enough from the root, then it is too expensive to backtrack.
\end{enumerate}
In order to show (1), the authors employ a bootstrap argument. Firstly they show an `easier' version of their Lemma 3.1 which guarantees that the walker reaches distance $\sqrt{\log n}$ in $n$ steps with high probability. Then, they leverage this fact to show that once the walker has reached distance  $\sqrt{\log n}$, it is more likely that it increases its distance by another factor than going back to the root. This argument is then iterated in the proper scale to show that if the walker reaches distance $\log^M n$ in $n $ steps with high enough probability, then it reaches distance $\log^{M+1/2}$ in $n$ steps as well.

The initial step in the bootstrap argument is Proposition 3.3. The argument relies only on the fact that at each step $t$, regardless the current tree and the position of the walker, the walker always have probability at least $p/2$ of jumping down. The worst case is the one in which the walker is on a leaf at time $t$, then it adds a leaf to its position with probability $p$ and jump to it with probability $1/2$. With this bounded away from zero probability and enough time, one can see the walker jumping down $\sqrt{\log n}$ times in a row.

Notice that if we have a family of $Q$-TBRW with $Q \in \mathcal{Q}_\kappa$, this condition is satisfied, since in this case $P(\xi_n \ge 1) \ge \kappa$ for all $n$. Thus, our family of TBRW can increase its distance by one at any time with probability at least $\kappa/2$.

Once Proposition 3.3 is proved, the authors then show Proposition 3.4, which states that if the walker reaches distance $\log^M n$ for some $M$ in $n$ steps with high enough probability, then it reaches distance $\log^{M+1/2} n $ in the same time window also with high enough probability.

The proof of Proposition 3.3 relies on Claim 3.5 and Lemma 3.6. The proof of Claim 3.5 again relies on the fact that the walker can always increase its distance from the root with probability at least $p/2$, which is still true in our case by just replacing $p$ by $\kappa$. The main idea is that every time the walker reaches the bottom of the tree, it has this bounded away from zero probability of adding a leaf to the bottom and jumping to it, which is uniformly true in our case.

More formally, the proof of Claim 3.5 is done by bounding from above three different terms, see (3.5), (3.6) and (3.7). The upper bound given at (3.5) follows from Proposition 3.3, which depends on $\kappa$ only. As for (3.6) it also depends on $\kappa$ only since it depends on the probability of adding at least one leaf and jumping to it. Whereas (3.7) is given by a comparison with a simple random walk on a path of finite length.

The proof of Proposition 3.4 then follows a scaling argument which depends only on the bounds given by Claim 3.5 and Proposition 3.3, which by their turns depend only on $\kappa$. 

Step (2) is then shown in Section 4, where the authors give an upper on the probability of hitting far way vertices in a short period of time. The argument is a coupling with what they call Loop Process, which consists of a random walk on a line segment that at each step the walker adds a loop to its position with probability $p$. In Proposition 4.4 they show that there exists a coupling of the loop process with the Bernoulli instance of the TBRW such that if $y$ is a vertex closer to the root and at distance $\ell$ from the initial position $x_0$, then the TBRW always takes more time to reach $y$ than the loop process.

Observe that from our Lemma \ref{l:return}, it follows that the hitting time to a vertex closer to the root of $Q$-TBRW, with $Q \in \mathcal{Q}_\kappa$, is always larger than the same hitting of a Bernoulli instance of the TBRW with parameter $\kappa$. Thus, the upper bound on $\eta_y$ (their notation for hitting time to $y$) given at Corollary 4.5 in \cite{FIORR21} still holds, with the constant $C$ depending on $\kappa$ only. 

The proof of Lemma 5.1 is then done using the above results, which depend on $\kappa$ only. The reader might find instructive to see Remark 5.4 in that paper for a discussion on the dependence on $\kappa$ (p in their case) of the constants in Lemma 5.1.

Finally, this way, the coupling with a right-biased random walk $S$ can be done for any $Q$-TBRW in our family with the same constants $r$ and drift by simply using $\kappa/2$ as the lower bound on the probability of increasing the distance by $1$ in one step.
    
\end{proof}

\section{The Renewal Structure: Proof of Theorem \ref{t:renew}}\label{s:renewalstructure}
In this section we show the construction of the renewal structure for the TBRW under the i.i.d.
setting. The reader will notice that the i.i.d. hypothesis is crucial here. The reason relates to our earlier remark about Markov properties in Section \ref{sss:markov}. When the TBRW is shifted, the leaf process is shifted as well. Thus, without i.i.d. hypothesis, we should not expect the TBRW to renewal since at each ``epoch'' we would observe a TBRW under a different leaf process.

The reader might find helpful to recall the definition of $\{\renew{k}\}_k$ given at \eqref{def:renew} and \eqref{def:renewk} at Page \pageref{def:renew}. We begin by introducing a new notation. We let $T(f(x))$ be the subtree of $T$ rooted at $f(x)$ whose vertex set is formed by $f(x),x$ and all the descendants of $x$ in $T$. See Figure \ref{fig:f_tree} below.
\begin{figure}[ht]
    \centering
    \includegraphics[width=0.26\linewidth]{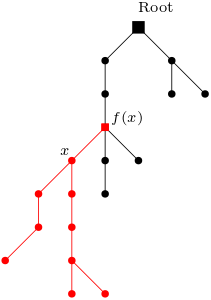}
    \caption{The subtree $T(f(x))$ in red.}
    \label{fig:f_tree}
\end{figure}

That is, to obtain $T(f(x))$, we prone $T$ by removing all the branches leaving~$f(x)$, expect the one containing $x$. When $x$ is the root $o$ of $T$, we put $T(f(o)) = T$. We also let $\mathcal{G}_k$ be the following  $\sigma$-algebra
\begin{equation}\label{def:g1}
    \mathcal{G}_k := \sigma\left( \renew{1}, \dots, \renew{k}, \Lbrace \left( T_n, X_n \right)\Rbrace_{n\le \renew{k}} \right).
\end{equation}
And for any $j \in \mathbb{N}$, we will let $R_{X_{j-1},j} $ be the first return to $X_{j-1}$ after time $j$, that is,
\begin{equation}\label{def:rj}
    R_{X_{j-1},j} := H_{X_0}^+\circ \theta_{j-1} + j,
\end{equation}
where $H_x^+ = \inf\{n > 0 \; : \; X_n = x\}$.
In order to prove Theorem \ref{t:renew} we will need two auxiliary results.
\begin{lemma}\label{l:decomp} For all $k,j \in \mathbb{N}$, it follows that there exists $B_{j} \in \mathcal{F}_{j}$, such that
    $$
        \{\renew{k} = j \} = B_{j}\cap \{\degr{T_j}{X_j} = 1\} \cap \{ R_{X_{j-1},j} = \infty \}.
    $$
\end{lemma}
\begin{proof} The proof follows by induction on $k$. We start with $\renew{1}$. Figure \ref{fig:twotaus} will serve as a support for the argument.
\begin{figure}[h]
    \centering
    \includegraphics[width=0.4\linewidth]{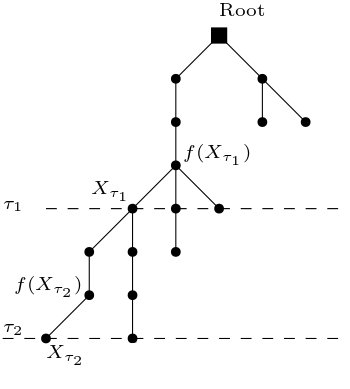}
    \caption{Two regeneration times.}
    \label{fig:twotaus}
\end{figure}

To see that the lemma holds for $\renew{1}$, notice that in order for $\renew{1}$ to be equal to~$j$, $X$ reaches a certain level below its initial position $x$ for the first time at time $j$ by jumping to a leaf. Before time $j$ either $X$ has reached a level below $x$ but not jumping to a leaf or it reached a further level jumping to a leaf but eventually visited the father of that leaf. In the later case, due to the tree structure, if $\renew{1} = j$, $X$ must visit the father of all leafs it visited before time $j$ also before time $j$, otherwise, $X$ would be forced to visit $f(X_{j})$ and then $\renew{1} > j$, see Figure \ref{fig:twotaus}. This means that for any $j \in \mathbb{N}$, we have that
\begin{equation}\label{eq:decompr1}
\{\renew{1} = j \} = B_{j} \cap \{\degr{T_j}{X_j} = 1\} \cap \{ R_{X_{j-1},j} = \infty \}. 
\end{equation}
To avoid unnecessary clutter with the notation, we will show how to extend the result from $\renew{1}$ to $\renew{2}$. Fix $i,j \in \mathbb{N}$, with $i < j$. Using \eqref{eq:decompr1}, we have that 
$$
\{\renew{1} = i, \renew{2} = j\} = B_{i}\cap\{\degr{T_i}{X_i} = 1, R_{X_{i-1},i} = \infty ,\renew{2} = j\}.
$$
Now, notice that in order for $\renew{2}$ to be equal to $j$ knowing that $\renew{1}=i$, either $X$ visits some leaf between times $i+1$ and $j-1$ or not. But since after time $j$, $X$ does not visits $f(X_j)$ any more, all visits to father of leafs visited between steps $i+1$ and $j-1$ must occur in the same time window, see Figure \ref{fig:twotaus}. Thus we can rewrite the event $\{\renew{1} = i, \renew{2} = j\}$ as
$$
 B_{i}\cap\{\degr{T_i}{X_i} = 1, R_{X_{i-1},i} = \infty \} \cap B_{j} \cap\{\degr{T_j}{X_j} = 1, R_{X_{j-1},j} = \infty \}.
$$
Finally, notice that the following identity of events holds
\begin{align*}
    & \{ R_{X_{i-1},i} = \infty, \degr{T_j}{X_j} = 1, R_{X_{j-1},j} = \infty \}  & \\
      & \quad = \{ R_{X_{i-1},i} > j, \degr{T_j}{X_j} = 1, R_{X_{j-1},j} = \infty \}        
\end{align*}
due to the tree structure of $T$ and the fact that $i < j$, again see Figure \ref{fig:twotaus}. Noting that the event $\{ R_{X_{i-1},i} > j\}$ is $\mathcal{F}_j$-measurable, we finally obtain 
$$
\{\renew{1} = i, \renew{2} = j\} = B'_{i,j}\cap \{\degr{T_j}{X_j} = 1, R_{X_{j-1},j} = \infty \},
$$
where $B'_{i,j}$ depends on $i$ but is $\mathcal{F}_j$-measurable since $i$ must be smaller than $j$. Taking the union over $i$ smaller than $j$, we show that the lemma holds for $\renew{2}$ as well. The general case is done by induction.
\end{proof}
The next result is the core of the renewal structure and can be seen as the renewal theorem itself as will become clear later. 
\begin{proposition}\label{p:renew} Consider a $Q$-TBRW with $Q\in \mathcal{Q}_\kappa$, for some positive $\kappa$, and starting from a finite initial state $(T_0,x_0)$. Then, the following holds $P_{T_0,x_0; Q}$-almost surely for all natural $k$:
\begin{enumerate}[label = (\alph*)]
    \item $\renew{k}$ is finite;
    \item 
    \begin{equation*}
        \begin{split}
             P_{T_0,x_0; Q} \left( \Lbrace \left( T_{t+\renew{k}}(f(X_{\renew{k}})), X_{t+\renew{k}} \right)\Rbrace_{t\ge 0} \in  \cdot \middle | \mathcal{G}_k \right) \\
             = P_{Q}  \left( \Lbrace \left( T_{t}, X_{t} \right)\Rbrace_{t\ge 0} \in  \cdot \; \middle | \; H_o = \infty \right).
        \end{split}
    \end{equation*}    
 
\end{enumerate}

\end{proposition}
\begin{proof} The proof will follow by induction on $k$. We will first show that $\renew{k}$ is finite almost surely, then we show the identity given at $(b)$.

The case $k=1$ is done already. Theorem 1.3 of \cite{IRVZ22} guarantees that the TBRW is ballistic, consequently transient, for any finite initial state $(T_0,x_0)$. Since our initial state is finite, after a certain time, the walker will reach new levels by jumping to a leaf, because it must construct a new structure to reach distances larger than the height of $T_0$. Being transient, every time the walker reaches a new level through a leaf, it has a bounded away from zero chance of never returning to the father of that leaf. If it does return, we can shift the process and use Theorem 1.3 again and wait until it reaches a futher level jumping to a leaf, when it has the same chance to not return to the father of that leaf. This is enough to guarantee that~$\tau_1$ is finite $\Pl{T_0,x_0; Q}$-almost surely. 
\begin{remark}
    The reader might be questioning themselves why the finiteness of $\renew{1}$ is not an direct consequence of Theorem \ref{t:gtailrenew}. Notice that in that theorem, we obtained a bound for the tail of $\renew{1}$ when the process starts from an edge. Thus, one should expect that we can transfer it to any initial state $(T_0,x_0)$. However, $(T_0,x_0)$ might be a state which is not accessible by the TBRW when it starts from an edge. For instance, consider $\xi_n \sim {\rm Ber}(1)$ for all $n$. When the TBRW starts from an edge, it is impossible for it to reach a pair $(T,x)$ which is a star centered in $x$, because the walker always adds a new vertex when it steps on a vertex. This way, we cannot leverage Theorem \ref{t:gtailrenew} when we start from this initial state. Of course we expect the result also to be true for any finite initial state at the cost of constants depending on the initial state too, but for our purposes, we won't need this degree of generality.
\end{remark} 
Now, suppose that we have successfully showed that $\renew{i}$ is finite almost surely for~$i\le k$. We will then show that $(b)$ holds for $\renew{k}$ and use it to show that this implies that $\renew{k+1}$ is $\Pl{T_0,x_0; Q}$-almost surely finite as well.\\

\noindent \underline{Showing (b) for $\renew{k}$.} To avoid clutter, we begin introducing the following shorthand
\begin{equation}\label{def:yt}
    Y_t(f(x)) := \left( T_{t}(f(x)), X_{t} \right).
\end{equation}
For the particular case $x = o$, we write $Y_t$ instead of $Y_t(f(o))$. Now, fix $A \in \mathcal{G}_{k}$ and notice that for each $j \in \mathbb{N}$, there exists $A_j \in \mathcal{F}_j$ such that
$$
A\cap \{\renew{k} = j \} = A_j\cap \{\renew{k} = j \}.
$$
Thus, 
\begin{equation}\label{eq:psum}
    \begin{split}
       P_{T_0,x_0; Q} \left( \Lbrace Y_{t+\renew{k}}(f(X_{\renew{k}})) \Rbrace_{t \ge 0}\in \cdot, A \right) \\ = \sum_{j}P_{T_0,x_0; Q} \left( \Lbrace Y_{t+\renew{k}}(f(X_{\renew{k}})) \Rbrace_{t \ge 0}\in \cdot, A_k, \renew{k} = j \right).
    \end{split}
\end{equation}
By Lemma \ref{l:decomp}, for each $j \in \mathbb{N}$ there exists $B_{j} \in \mathcal{F}_{j}$ such that
\begin{equation}\label{eq:decomptk}
    \{\renew{k} = j \} =  B_{j}\cap \{\degr{T_j}{X_j} = 1, R_{X_{j-1},j} = \infty \}.
\end{equation}
Writing $C_{i,j} = A_j \cap B_{j} \cap \{\degr{T_j}{X_j} = 1,X_j = x_i\}$, we have that 
\begin{equation}
    \mathds{1}_{\Lbrace \Lbrace Y_{t+j}(f(x_i)) \Rbrace_{t \ge 0}\in \cdot, R_{f(x_i),j} = \infty, C_{i,j}\Rbrace} = \mathds{1}_{\{\Lbrace Y_{t}(f(x_i)) \Rbrace_{t \ge 0}\in \cdot, H_{f(x_i)} = \infty\}} \circ \theta_j \cdot \mathds{1}_{C_{i,j}},
\end{equation}
$P_{T_0,x_0; Q} $-almost surely. Observing that $C_{i,j}$ is $\mathcal{F}_j$-measurable, by Simple Markov property, it follows that $P_{T_0,x_0; Q} \left(  \Lbrace Y_{t+j}(f(x_i)) \Rbrace_{t \ge 0}\in \cdot, R_{f(x_i),j} = \infty,C_{i,j}\right)$ equals
\begin{equation}\label{eq:ckj}
    \begin{split}
         E_{T_0,x_0; Q} \left[ P_{T_j,X_j; Q} \left[\Lbrace Y_{t}(f(x_i)) \Rbrace_{t \ge 0}\in \cdot, H_{f(x_i)} = \infty\right] \mathds{1}_{C_{i,j}} \right].
    \end{split}
\end{equation}
Notice that on the event $C_{i,j}$, we have that 
\begin{equation}
    P_{T_j,X_j; Q} \left[\Lbrace Y_{t}(f(x_i)) \Rbrace_{t \ge 0}\in \cdot, H_{f(x_i)} = \infty\}\right] = P_{ Q} \left( \Lbrace Y_{t} \Rbrace_{t \ge 0}\in \cdot, H_{o} = \infty\right),
\end{equation}
since the the leaf process $\xi$ is i.i.d., and on $C_{i,j}$ we have $X_j = x_i$ and  $\degr{T_j}{X_j} = 1$. This means on the event $C_{i,j}$, in both sides of the above identity, the process starts from a leaf. Moreover, in the event inside the probability on the LHS of the above identity, the walker does not walk over portions of $T_j$ beyond $f(x_i)$. Thus what we have is a $Q$-TBRW starting from a single edge that never visits the other tip of it. 

Figure \ref{fig:taukj} below helps us to see that starting from $(T_j,X_j)$ the evolution of $(T_t(f(x_i),X_t)$ can be coupled to a TBRW that starts from an edge, as long as the TBRW starting from $(T_j,X_j)$ does not visit the vertices in black. Notice that we need here that the leaf process is i.i.d.
\begin{figure}[h]
    \centering
    \includegraphics[width=0.3\linewidth]{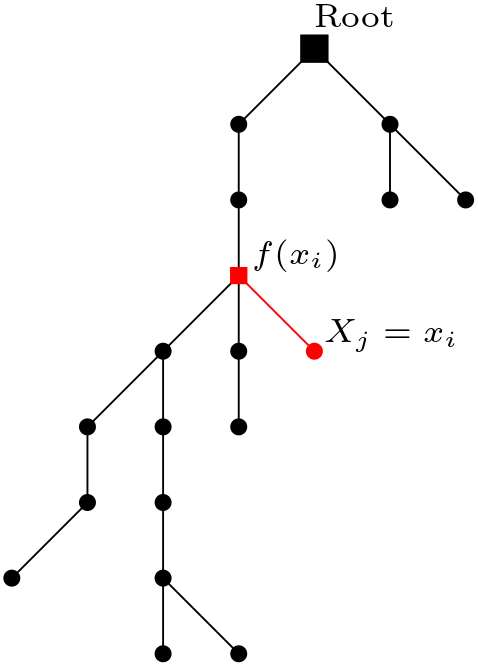}
    \caption{$\renew{k} = j$}
    \label{fig:taukj}
\end{figure}

This implies that $P_{T_0,x_0; Q} \left(  \Lbrace Y_{t+j}(f(x_i)) \Rbrace_{t \ge 0}\in \cdot,C_{i,j}, R_{f(x_i),j} = \infty\right)$ equals
$$
P_{Q} \left( \Lbrace Y_{t} \Rbrace_{t \ge 0}\in \cdot, H_{o} = \infty\right)\cdot P_{T_0,x_0; Q} (C_{i,j}).
$$
Summing over $i$, returning to \eqref{eq:psum} and recalling \eqref{eq:decomptk}, it follows that
\begin{equation}\label{eq:prod}
    \begin{split}
        P_{T_0,x_0; Q} \left( \Lbrace Y_{t+\renew{k}}(f(X_{\renew{k}})) \Rbrace_{t \ge 0}\in \cdot, A \right) \\
        = P_{Q} \left( \Lbrace Y_{t} \Rbrace_{t \ge 0}\in \cdot, H_{o} = \infty\right) \sum_j P_{T_0,x_0; Q} (A_j, B_{j}, \degr{T_j}{X_j} = 1).        
    \end{split}
\end{equation}
On the other hand, using the Simple Markov property backwards, it follows that
\begin{equation*}
    \begin{split}
        P_{ Q} (H_o = \infty)P_{T_0,x_0; Q} (A_j,B_{j},\degr{T_j}{X_j} = 1) = P_{T_0,x_0; Q} (A,\renew{k} = j) .
    \end{split}
\end{equation*}
Plugging the above identity into \eqref{eq:prod}, summing over $j$ and using our inductive hypothesis that $\renew{k}$ is finite almost surely, yields
\begin{equation*}
    P_{T_0,x_0; Q} \left( \Lbrace Y_{t+\renew{k}}(f(X_{\renew{k}})) \Rbrace_{t \ge 0}\in \cdot, A \right) = P_{ Q} \left( \Lbrace Y_{t} \Rbrace_{t \ge 0}\in \cdot \; \middle | \; H_{o} = \infty\right)P_{T_0,x_0; Q} \left( A \right),
\end{equation*}
which is equivalent to show $(b)$ for $\renew{k}$.\\

\noindent \underline{Showing $\renew{k+1}$ is finite.} The key observation for this step is the identity below
\begin{equation}\label{eq:tk1}
    \renew{k+1}\left( \Lbrace \left( T_n,X_n\right)\Rbrace_n \right)= \renew{1}\left( \Lbrace \left( T_{n+\renew{k}}(f(X_{\renew{k}})),X_{n+\renew{k}}\right)\Rbrace_n \right) + \renew{k}\left( \Lbrace \left( T_n,X_n\right)\Rbrace_n \right).
\end{equation}
That is, when the process is shifted by $\renew{k}$, we know that $X$ does not visit $f(X_{\renew{k}})$ again. Consequently, in terms of regeneration after $\renew{k}$, we can prone the trees at $f(X_{\renew{k}})$ since the process will not see portions of the environment beyond $X_{\renew{k}}$, being restricted to what it is being built on top of  $T_{\renew{k}}(f(X_{\renew{k}}))$. 

Moreover, by induction, it is enough to show that $\renew{1}\left( \Lbrace \left( T_{n+\renew{k}}(f(X_{\renew{k}})),X_{n+\renew{k}}\right)\Rbrace_n \right)$ is finite. In order to do so, we use $(b)$, which gives us that $\Lbrace \left( T_{n+\renew{k}}(f(X_{\renew{k}})),X_{n+\renew{k}}\right)\Rbrace_n$ is distributed as $\Lbrace \left( T_{n},X_{n}\right)\Rbrace_n$ conditioned on $\{H_o = \infty \}$. This implies that the random time $\renew{1}\left( \Lbrace \left( T_{n+\renew{k}}(f(X_{\renew{k}})),X_{n+\renew{k}}\right)\Rbrace_n \right)$ is distributed as $\renew{1}$ conditioned on $\{H_o = \infty\}$. Since $\renew{1}$ is finite almost surely, it follows that $\renew{k+1}$ is finite almost surely as well, which concludes the proof.
\end{proof}
Once we have Proposition \ref{p:renew}, the proof of Theorem \ref{t:renew} is a mere observation.
\begin{proof}[Proof of Theorem \ref{t:renew}] The proof is a consequence of Proposition \ref{p:renew} combined with Equation \eqref{eq:tk1}. Firstly, recall that $\renew{k+1} - \renew{k} = \tau_1 \circ \theta_{\tau_k}$. Moreover, observe that the value of $\tau_1$ when the process is shifted by $\tau_k$ does not depend on portions of tree constructed before $\renew{k}$, except from $f(X_{\renew{k}})$. That is,
$$
\renew{1} \circ \theta_{\renew{k}}\left( \Lbrace \left( T_{n},X_{n}\right)\Rbrace_n \right)  = \renew{1}\left( \Lbrace \left( T_{n+\renew{k}}(f(X_{\renew{k}})),X_{n+\renew{k}}\right)\Rbrace_n \right)
$$
This combined to Proposition \ref{p:renew}, implies that $\renew{k+1} - \renew{k}$, for $k\ge 1$, is distributed as $\renew{1}$ conditioned on $\{H_o = \infty\}$.
Same observation holding for $\dist{X_{\renew{k+1}}}{o} - \dist{X_{\renew{k}}}{o}$, with $k \ge 1$. The independence also follows immediately by conditioning on the proper $\mathcal{G}_j$.
    
\end{proof}

\section{Limit Theorems: Proof of Theorems \ref{t:lln}, \ref{t:lil} and \ref{t:clts}}\label{s:proofthm1}
Theorems \ref{t:lln}, \ref{t:lil} and \ref{t:clts} are consequence of our renewal structure given by Theorem \ref{t:renew}, the very light tail of $\renew{1}$ given by Theorem \ref{t:gtailrenew} and Theorem 1 in \cite{GW93}.

In \cite{GW93} the authors address the question of existence of limit theorems for cumulative processes associated to regenerative processes in a general setup. They establish necessary and sufficient conditions for the cumulative process (the distance process in our context) to satisfy limit theorems. Thus, in order to show Theorems \ref{t:lln}, \ref{t:lil} and \ref{t:clts}, we need to show that conditions given by Theorem 1 in \cite{GW93} are satisfied.

To facilitate the application of their results, we provide Table \ref{tab:notation} below which contains our notation under their framework. The reader will notice that the existence of limit theorems for the cumulative process involve conditions on the moments of $W_1(f)$ and $W_1(f_c)$, which in our context, is the maximum distance the walker reaches after $\renew{1}$ and before $\renew{2}$. Since the maximum of such excursion is bounded by the length of the interval $[\renew{1}, \renew{2}]$, our results guarantee that we have an extremely good control on the tail of $W_1$ and $W_1(f_c)$.
\begin{table}[h]
    \centering
    \begin{tabular}{ll}
    \hline
    \multicolumn{1}{|l|}{ Notation in \cite{GW93}} &   \multicolumn{1}{l|}{Our notation} \\ \hline
    \multicolumn{1}{|l|}{$X(s)$} & \multicolumn{1}{l|}{$(T_{\lfloor s \rfloor}, X_{\lfloor s \rfloor})$} \\ \hline
     \multicolumn{1}{|l|}{$f(X(s))$} & \multicolumn{1}{l|}{$\dist{X_{\lceil s \rceil}}{o} - \dist{X_{\lfloor s \rfloor}}{o}$} \\ \hline
     \multicolumn{1}{|l|}{$T(i)$} & \multicolumn{1}{l|}{$\renew{i+1}$} \\ \hline
      \multicolumn{1}{|l|}{$\tau_i$} & \multicolumn{1}{l|}{$\renew{i+1}-\renew{i}$} \\ \hline
      \multicolumn{1}{|l|}{$C(t)$} & \multicolumn{1}{l|}{$\dist{X_{\lceil t \rceil}}{o} - \dist{X_0}{o}$} \\ \hline
      \multicolumn{1}{|l|}{$Y_i(f)$} & \multicolumn{1}{l|}{$\dist{X_{\renew{i+1}}}{o} - \dist{X_{\renew{i}}}{o}$} \\ \hline
      \multicolumn{1}{|l|}{$\alpha$} & \multicolumn{1}{l|}{$v = v(Q)$} \\ \hline
      \multicolumn{1}{|l|}{$\beta$} & \multicolumn{1}{l|}{$\sigma^2 = \sigma^2(Q)$} \\ \hline
      \multicolumn{1}{|l|}{$Y_i(f_c)$} & \multicolumn{1}{l|}{$\dist{X_{\renew{i+1}}}{o} - \dist{X_{\renew{i}}}{o} - v\cdot(\renew{i+1} - \renew{i})$} \\ \hline
      \multicolumn{1}{|l|}{$W_i(f)$} & \multicolumn{1}{l|}{$\sup_{0\le s \le \renew{2} - \renew{1}}|\dist{X_{\lceil \renew{i} + s\rceil}}{o} - \dist{X_{\renew{i}}}{o}|$} \\ \hline
      \multicolumn{1}{|l|}{$S_n$} & \multicolumn{1}{l|}{$\dist{X_{ \renew{n+1}}}{o} - \dist{X_{\renew{1}}}{o}$} \\ \hline
    \end{tabular}
    \caption{Our notation under Glynn and Whitt's framework.}
    \label{tab:notation}
\end{table}

We are now ready to prove the main results of this section.
\begin{proof}[Proof of Theorems \ref{t:lln}, \ref{t:lil} and \ref{t:clts}] We begin noticing that by triangle inequality it follows that 
\begin{equation}
    \begin{split}
        W_1(f_c) & \le v\cdot(\renew{2} - \renew{1}) +  \sup_{0\le s \le \renew{2} - \renew{1}}|\dist{X_{\lceil \renew{1} + s\rceil}}{o} - \dist{X_{\renew{1}}}{o}| \\ 
        & \le (1+v)(\renew{2} - \renew{1}).
    \end{split}
\end{equation}
Then the proof follows from Theorem 1 in \cite{GW93} and Theorems \ref{t:renew} and \ref{t:gtailrenew}. Since in this case all the moments of $W_1, W_1(f_c), Y_1$ and $Y_1(f_c)$ are finite  
\end{proof}

\section{The Speed is Continuous: Proof of Theorem \ref{t:1}}
We now move to the proof of Theorem \ref{t:1} which states that the speed of the TBRW is a continuous function on the space of probability distributions over $\mathbb{N}$ with the total variation distance. The proof will involve essentially all the results we have proved so far.
\begin{proof}[Proof of Theorem \ref{t:1}: The speed is continuous.] We begin recalling the expression for $v$. By Theorem \ref{t:lln}, fixed $\mu \in \pspace$, $v(\mu)$ is given by
\begin{equation}\label{eq:v}
    v(\mu) = \frac{E_{\mu}\left[ \; \dist{X_{\renew{1}}}{o} \; \middle | \;  H_o = \infty \right]}{E_{ \mu}\left[ \; \renew{1} \; \middle | \; H_o = \infty \right]} = \frac{E_{ \mu}\left[ \; \dist{X_{\renew{1}}}{o} \; ; \;  H_o = \infty \right]}{E_{\mu}\left[ \; \renew{1} \; ; \; H_o = \infty \right]}.
\end{equation}
We want to show that $v$ is continuous. In order to do that, we will show that the function $\mu \mapsto E_{\mu}\left[ \; \renew{1} \; ; \; H_o = \infty \right]$ is continuous. The continuity of the numerator in \eqref{eq:v} will follow similarly. 

Throughout the proof, $\{Q_j\}_j$ denotes a sequence of distributions in $\pspace$ which converges to $Q$ in the total variation distance. We will first assume that $Q \in \mathcal{Q}_\kappa$, for some $\kappa >0$. The case $Q = \delta_0$ will be handled separately. Notice that in the first case, we can assume w.l.o.g that the whole sequence together with $Q$ belongs to $\mathcal{Q}_\kappa$ for some possibly different $\kappa$. \\

\noindent \underline{\it Case 1. $Q\neq \delta_0$} As our first step, we will need a way to compare the difference between the regeneration time of two instances of the TBRW generated by $Q_j$ and $Q$. For each $j$ we will construct a coupling $\cPl_{Q_j,Q}$ between the two TBRW similar to the coupling construct in the proof of Lemma \ref{l:return}, for this reason we will skip some details.

We first couple the leaf processes $\{\xi_n^{(Q_j)}\}_j$ and $\{\xi_n^{(Q)}\}_j$. For each $n$, we couple $\xi_n^{(Q_j)}$ and $\xi_n^{(Q)}$ using the optimal coupling described in \cite{LevinPeresWilmer2006} at Section 4.2. Let $p$ be 
\begin{equation}\label{def:p}
   p_j := 1-\dtv{Q_j}{Q}
\end{equation}
and $\{Y_n\}_n$ an i.i.d Bernoulli sequence of parameter $p$. Using $Y_n$, we can couple $\xi_n^{(Q_j)}$ and $\xi_n^{(Q)}$ in way that $\xi_n^{(Q_j)} = \xi_n^{(Q)}$ if and only if $Y_n = 1$. Moreover, since we are coupling using an independent source of randomness provided by $Y_n$'s, the sequence $\{\xi_n^{(Q_j)}\}_j$ (resp. $\{\xi_n^{(Q)}\}_j$) is i.i.d.

Now, let $\zeta$ be the following stopping time
\begin{equation}
    \zeta_j := \inf \{ n >0 \; : \; Y_n = 0\}.
\end{equation}
In words, $\zeta_j$ is the first time that $\xi_n^{(Q_j)} \neq \xi_n^{(Q)}$. Notice that for each $j$, $\zeta_j \sim {\rm Geo}(\dtv{Q_j}{Q})$. Thus, up to time $\zeta_j$ we can simply move both TBRW's together on the same tree, since they start at the same position and at each step $n<\zeta_j$ add the same number of vertices to the tree. After $\zeta_j$ we can move the walkers independently and label their vertices in way that the vertices added by one process are not visible to other, as we did in the proof of Lemma \ref{l:return}.

Our next step is to control $|\renew{1}^{(Q_j)} - \renew{1}^{(Q)}|$. For any given $\varepsilon$, observe that the following identity of events holds true
\begin{equation}\label{eq:empty}
    \Lbrace |\renew{1}^{(Q_j)}\mathds{1}\{H_o^{(Q_j)} = \infty\} - \renew{1}^{(Q)}\mathds{1}\{H_o^{(Q)} = \infty\}|>\varepsilon, \; \renew{1}^{(Q_j)} \vee  \renew{1}^{(Q)}  < \zeta_j \Rbrace = \emptyset.
\end{equation}
Indeed, suppose towards a contradiction and without loss of generality that $\renew{1}^{(Q_j)} = n_1$ and $\renew{1}^{(Q)} = n_2$ with $n_1 < n_2 < \zeta_j$. Because both processes evolve together until $\zeta_j$, we must have that $X^{(Q)}_{n_2}$ belongs to subtree of $X^{(Q_j)}_{n_1}$. Moreover, in order for $\renew{1}^{(Q)} = n_2 > n_1$, $X^{(Q)}$ must visit the father of $X^{(Q)}_{n_1}$ after $\zeta_j$. However, this implies visiting the father of $X^{(Q)}_{n_2}$ after $\zeta_j$, which by its turn implies that $\renew{1}^{(Q)} > \zeta_j$. Thus, on the event $\{\renew{1}^{(Q_j)} \vee  \renew{1}^{(Q)}  < \zeta_j \}$, we have that $\renew{1}^{(Q_j)} = \renew{1}^{(Q)}$. 

Now, suppose w.l.o.g that $H_o^{(Q_j)} = \infty$ and $H_o^{(Q)} < \infty$. Again, since the walkers walk together until $\zeta_j$, in this case $X^{(Q)}$ must visit the root after $\zeta_j$. However, this automatically implies that $\renew{1}^{(Q)} > \zeta_j$, since the $Q$-walker must visit $f(X^{(Q)}_{\renew{1}})$ in its way back to the root. This contradicts the assumption that $\renew{1}^{(Q)} < \zeta_j$.

Calling $Z_j = \renew{1}^{(Q_j)}\mathds{1}\{H_o^{(Q_j)} = \infty\}$ and $Z = \renew{1}^{(Q)}\mathds{1}\{H_o^{(Q)} = \infty\}$ to avoid clutter, \eqref{eq:empty} and union bound leads us to the following upper bound
\begin{equation}\label{ineq:0}
    \begin{split}
    \cPl_{Q_j,Q}\left( |Z_j - Z|>\varepsilon \right) \le
    \cPl_{Q_j,Q}\left( \renew{1}^{(Q)} > \zeta_j \right) + \cPl_{Q_j,Q}\left( \renew{1}^{(Q_j)} > \zeta_j \right).
    \end{split}
\end{equation}
We bound the second term of the RHS of the above inequality first. Again by union bound,
\begin{equation}\label{ineq:1}
    \begin{split}
        \cPl_{Q_j,Q}\left( \renew{1}^{(Q_j)} > \zeta_j \right) & \le\cPl_{Q_j,Q}\left( \zeta_j \le \frac{1}{\sqrt{\dtv{Q_j}{Q}}} \right) \\
        & \quad  + \cPl_{Q_j,Q}\left( \renew{1}^{(Q_j)} > \frac{1}{\sqrt{\dtv{Q_j}{Q}}}  \right).
    \end{split}
\end{equation}
Recall that by the definition of $\zeta_j$, it follows a geometric distribution of parameter~$\dtv{Q_j}{Q}$. Thus, combining this with the fact $1-x \ge e^{-3x/2}$, for $x<1/2$, and that $1-x \le e^{-x}$, yields
\begin{equation}\label{ineq:2}
    \cPl_{Q_j,Q}\left( \zeta_j \le \frac{1}{\sqrt{\dtv{Q_j}{Q}}} \right) \le 3\sqrt{\dtv{Q_j}{Q}}/2.
\end{equation}
As for the second term of the RHS of \eqref{ineq:1}, we apply Theorem \ref{t:gtailrenew}, which ensures the existence of two constants $C$ and $C'$, depending on $\kappa$ only, such that 
\begin{equation}
    \cPl_{Q_j,Q}\left( \renew{1}^{(Q_j)} > \frac{1}{\sqrt{\dtv{Q_j}{Q}}}  \right) \le Ce^{-C'/\sqrt{\dtv{Q_j}{Q}}}.
\end{equation}
Since the above bound depends only on $\kappa$, and \eqref{ineq:2} depends on $\dtv{Q_j}{Q}$ only, reproducing the same reasoning to $\cPl_{Q_j,Q}\left( \renew{1}^{(Q)} > \zeta_j \right)$ and returning to \eqref{ineq:0} yields
\begin{equation*}
    \begin{split}
         \cPl_{Q_j,Q}\left( |Z_j  - Z|>\varepsilon \right) \le 3\sqrt{\dtv{Q_j}{Q}} + 2Ce^{-C'/\sqrt{\dtv{Q_j}{Q}}},
    \end{split}
\end{equation*}
for all $\varepsilon>0$ and $j$. Using that the renewal times are integer numbers, it follows that 
\begin{equation}\label{ineq:probconv}
    \begin{split}
         \cPl_{Q_j,Q}\left( Z_j \neq Z \right)  \le 3\sqrt{\dtv{Q_j}{Q}} + 2Ce^{-C'/\sqrt{\dtv{Q_j}{Q}}}.
    \end{split}
\end{equation}
Now, notice that the uniform bound provided by Theorem \ref{t:gtailrenew} gives us that 
\begin{equation}
    \sup_{Q\in \mathcal{Q}_\kappa}\El{Q}[\renew{1}^2] = M < \infty.
\end{equation}
Thus, using the above bound, the triangle inequality and Cauchy-Schwarz inequality we deduce the following upper bound
\begin{equation*}
    \begin{split}
        \mathbb{E}_{Q_j,Q}\left[ \left | \renew{1}^{(Q_j)}\mathds{1}\{H^{(Q_j)}_o = \infty\} - \renew{1}^{(Q)}\mathds{1}\{H^{(Q)}_o = \infty\}\right |\right] \le 2M\cPl_{Q_j,Q}\left( Z_j\neq Z \right),
    \end{split}
\end{equation*}
which goes to zero as $j$ goes to infinity by virtue of \eqref{ineq:probconv}. Finally, by Jensen's inequality we can show that
$$
\left | E_{Q_j}\left[ \renew{1}\mathds{1}\{H_o = \infty\} \right] - E_Q\left[ \renew{1}\mathds{1}\{H_o = \infty\} \right]\right| \stackrel{j \to \infty}{\longrightarrow} 0,
$$
which shows that the denominator of $v$ in \eqref{eq:v} is continuous. To extend the result to the numerator, recall that on the event $\{\renew{1}^{(Q_j)} \vee  \renew{1}^{(Q)}  < \zeta_j \}$, $\renew{1}^{(Q_j)} = \renew{1}^{(Q)}$ and $\mathds{1}\{H_o^{(Q_j)} = \infty\} = \mathds{1}\{H_o^{(Q)} = \infty\}$. Thus, the probability of the event 
$$
\{\dist{X^{(Q_j)}_{\renew{1}}}{o}\mathds{1}\{H_o^{(Q_j)} = \infty\} \neq \dist{X^{(Q)}_{\renew{1}}}{o}\mathds{1}\{H_o^{(Q)} = \infty\} \}
$$
satisfies the same upper bound given in \eqref{ineq:0}. Also notice that $\dist{X_{\renew{1}}}{o} \le \renew{1}$. Thus the exact same argument works for the numerator. This shows that $v$ is continuous at any $Q \neq \delta_0$.
\\

\noindent \underline{\it Case 2. $Q = \delta_0$.} We conclude the proof showing that $v$ is also continuous at $Q = \delta_0$. The first observation we make is that $v(\delta_0) = 0$, since in this case the TBRW is walking over a fixed graph, consequently its distance from the root is bounded by the height of $T$. The second one is that, for any given distribution $\mu$ over $\mathbb{N}$ and time $n$ the following bound holds $\Pl{T,x;\mu}$-almost surely
\begin{equation*}
    \frac{\dist{X_n}{o}}{n} \le \frac{\dist{x}{o} + \sum_{k=1}^n \mathds{1}\{\xi_k \ge 1\}}{n}.
\end{equation*}
The equality holds when every time the walker adds at least one leaf, it jumps to one of the leaves right away.

By Theorem \ref{t:lln} and the Law of Large Numbers, it follows that for any $\mu$
\begin{equation}
    v(\mu) = \lim_{n\to \infty} \frac{\dist{X_n}{o}}{n} \le \mu(\{1,2,\dots,\}).
\end{equation}
On the other hand, if $\{Q_j\}_j$ is converging to $\delta_0$ in the total variation distance, then in particular $Q_j(\{1,2,\dots,\})$ converges to $\delta_0(\{1,2,\dots,\}) = 0$ as $j$ goes to infinity and this shows that $v(Q_j)$ converges to $v(\delta_0)$, which concludes the proof.
\end{proof}

\section{Non-identically distributed leaf processes}\label{s:noniid} 
In this section we investigate the speed when the identically distributed hypothesis on the leaf process is dropped. In this case, even under the uniformly elliptic condition \ref{def:UE}, the TBRW might not have a well-defined speed, that is, it might not satisfy a LLN. We will start showing the proposition below, which states that if the leaf process is not identically distributed but converges almost surely to a limit random variable $\xi_\infty$, then the speed at time $n$ converges has a (possibly random) limit speed. We then end the section constructing a ballistic TBRW which whose speed does not converge to any (random) limit speed.
\begin{proposition}[Limit speed for non-identically distributed leaf processes]\label{p:noiidlln} Consider a TBRW starting from a finite initial state $(T,x)$ with an independent leaf process $\xi = \{\xi_n\}_n$ that converges almost surely to $\xi_{\infty}$. Then, there exists a nonnegative random variable $S$ such that
$$
\lim_{n \to \infty} \frac{\dist{X_n}{o}}{n} = S, \; P_{T,x; Q}\text{-almost surely}.
$$
If $P(\xi_{\infty} \ge 1) = 1$, then $S$ is positive almost surely.
\end{proposition}
The proof will follow from a technical lemma stated below. But before we state it, let us make a few comments about the hypothesis we are working under and introduce some notation. We are assuming independence and that $\xi_n$ converges almost surely to $\xi_{\infty}$. Moreover, if $N$ is the following random time
\begin{equation}\label{def:N}
    N := \inf\{ n \ge 1 \; : \; \xi_m = \xi_\infty, \; \forall m \ge n \},
\end{equation}
then, by the fact that all random variables are discrete, $N$ is finite almost surely. 

For a fixed initial state $(T,x)$, we will make use of the notation $\widetilde{P}_{T,x; \xi_\infty}$ for the law of a TBRW whose sequence $\xi_n = \xi_\infty$ for all $n$. That is, in our first step, we sample a natural number $k$ according to $\xi_\infty$ and then we add $k$ new vertices at each step with probability $1$. In other words, conditioned on $\{\xi_\infty = k\}$, we have a TBRW with sequence $\xi_n \equiv k$ for all $n$. We introduce the notation $\widetilde{P}_{T,x; k}$ meaning that
$$
\widetilde{P}_{T,x; k}(\cdot) = \widetilde{P}_{T,x; \xi_\infty}( \; \cdot \; | \; \xi_\infty = k).
$$
We also let $\mathcal{G}$ be the $\sigma$-algebra defined as 
\begin{equation}
    \mathcal{G} := \sigma \left( \xi_\infty, N, \{(T_i,X_i)\}_{i=0}^N \right).
\end{equation}
That is, $\mathcal{G}$ contains the information given by $\xi_\infty, N$ and the trajectory of the TBRW with leaf process $\{\xi_n\}_n$ up to time $N$. The next lemma can be interpreted as a Strong Markov Property for the random time $N$.
\begin{lemma}\label{l:markovN}Consider a TBRW starting from $(T,x)$ with an independent leaf process $\xi = \{\xi_n\}_n$ which converges to $\xi_\infty$ almost surely. Then,
\begin{equation*}
    \Pl{T,x; \xi}\left( \{(T_{N+n},X_{N+n})\}_n \in \cdot  \; \middle | \;  \mathcal{G}\right) = \widetilde{P}_{T_N,X_N;\xi_\infty} (\{(T_n,X_n)\}_n \in \cdot ).
\end{equation*}
    
\end{lemma}
\begin{proof} To avoid clutter, let us write $Y_n := (T_n,X_n)$ and $(y)_0^j$ for a generic deterministic vector whose $i$th entry is a deterministic pair $(G,v)$, where $G$ is a finite rooted tree and $v$ is one of its vertices. 

Now, let $A \in \mathcal{G}$ be an event of the form
\begin{equation}\label{def:A}
    A:= \Lbrace \xi_\infty = k, N = j, \{Y_i\}_{i=0}^N = (y)_0^j\Rbrace.
\end{equation}
And let $B_j$ and $C_j$ be the events
\begin{equation}\label{def:bj}
    B_j := \Lbrace \xi_j = \xi_{j+1} = \dots = k \Rbrace; \quad C_j := \Lbrace \exists i \le j, \xi_i \neq \xi_\infty \Rbrace. 
\end{equation}
Notice that $A$ can be rewritten as 
\begin{equation}\label{def:A2}
    A:= B_j \cap C_{j-1} \cap \Lbrace \{Y_i\}_{i=0}^j = (y)_0^j\Rbrace.
\end{equation}
Also notice that  $C_{j-1} \in \mathcal{F}_j$, where $\{\mathcal{F}_n\}_n$ is the canonical filtration for the TBRW with sequence  $\xi = \{\xi_n\}_n$. Then, by Simple Markov Property
\begin{equation}\label{eq:1}
    \begin{split}
        \Pl{T,x; \xi}\left( \{Y_{N+n}\}_n \in \cdot, A \right) &= \Pl{T,x; \xi}\left( \{Y_{j+n}\}_n \in \cdot, A \right) \\
        & = \El{T,x; \xi}\left[ \mathds{1}\{\{Y_{n}\}_n \in \cdot, B_1 \} \circ \theta_j \mathds{1}{\{C_{j-1}, \{Y_i\}_{i=0}^j = (y)_0^j\}}\right] \\
        & = \El{T,x; \xi} \left[\Pl{T_j,X_j; \xi^{(j)}}(\{Y_{n}\}_n \in \cdot, B_1)\mathds{1}{\{C_{j-1}, \{Y_i\}_{i=0}^j = (y)_0^j\}}\right] \\
        & = \El{T,x; \xi} \left[\Pl{y_j; \xi^{(j)}}(\{Y_{n}\}_n \in \cdot, B_1)\mathds{1}{\{C_{j-1}, \{Y_i\}_{i=0}^j = (y)_0^j\}}\right]
    \end{split}
\end{equation}
Now, notice that a TBRW conditioned on all $\xi_i = k$ is a TBRW which adds $k$ leaves at each step with probability $1$. Thus, 
\begin{equation}\label{eq:yb}
    \begin{split}
        \Pl{y_j; \xi^{(j)}}(\{Y_{n}\}_n \in \cdot, B_1) &= \Pl{y_j; \xi^{(j)}}(\{Y_{n}\}_n \in \cdot, \xi_1 = \xi_2 = \dots = k) \\
        &= \widetilde{P}_{y_j; k}(\{Y_{n}\}_n \in \cdot)\Pl{y_j; \xi^{(j)}}( \xi_1 = \xi_2 = \dots = k)\\
        & = \widetilde{P}_{y_j; k}(\{Y_{n}\}_n \in \cdot)\Pl{y_j; \xi^{(j)}}( B_1).
    \end{split}
\end{equation}
On the other hand, we can reverse the Simple Markov property as follows
\begin{equation}
    \begin{split}
        \Pl{y_j; \xi^{(j)}}( B_1)\mathds{1}\{C_{j-1}, \{Y_i\}_{i=0}^j = (y)_0^j\} & = \El{T,x; \xi}\left[ B_1 \circ \theta_j  \mathds{1}\{C_{j-1}, \{Y_i\}_{i=0}^j = (y)_0^j\} \middle | \mathcal{F}_j\right]
    \end{split}
\end{equation}
Then plugging the above and \eqref{eq:yb} into \eqref{eq:1}, gives us
$$
\Pl{T,x; \xi}\left( \{Y_{N+n}\}_n \in \cdot, A \right) = \widetilde{P}_{y_j; k}(\{Y_{n}\}_n \in \cdot)\Pl{T,x; \xi}(A).
$$
Notice that on $A$, we have $(T_N,X_N) = y_j$ and $\xi_\infty = k$. This is enough to show the lemma.
\end{proof}
Now we can finally prove the proposition.
\begin{proof}[Proof of Proposition \ref{p:noiidlln}] Let $\{(\widetilde{T}_n, \widetilde{X}_n)\}_n$ be a TBRW with law $\widetilde{P}_{T,x;\xi_\infty}$,  for a finite initial state $(T,x)$. By Theorem \ref{t:lln}, it follows that there exists a random variable $S$ such that
\begin{equation}\label{eq:llns}
    \lim_{n\to \infty} \frac{\dist{\widetilde{X}_n,o}}{n} = S,
\end{equation}
$\widetilde{P}_{T,x; \xi_\infty}$-almost surely. The random variable $S$ does not depend on $(T,x)$ and is defined as $v_k$ on $\{\xi_\infty = k\}$, where $v_k$ is the speed of a TBRW that adds $k$ vertices with probability $1$ at each step. Notice that $S=0$ only on the event $\{\xi_\infty = 0\}$.

Finally, to extend \eqref{eq:llns} to the TBRW with sequence $\{\xi_n\}_n$, we leverage the asymptotic nature of the speed, which gives us that
$$
\mathds{1}\Lbrace \lim_{n\to \infty} \frac{\dist{X_n,o}}{n} = S \Rbrace = \mathds{1}\Lbrace \lim_{n\to \infty} \frac{\dist{X_n,o}}{n} = S \Rbrace \circ \theta_N,
$$
$\Pl{T,x; \xi}$-almost surely. Thus, by Lemma \ref{l:markovN}, we have that 
\begin{equation*}
    \Pl{T,x; \xi} \left( \lim_{n\to \infty} \frac{\dist{X_n,o}}{n} = S \;  \middle | \; \mathcal{G}\right) = \widetilde{P}_{T_N,X_N; \xi_\infty} \left( \lim_{n\to \infty} \frac{\dist{\widetilde{X}_n,o}}{n} = S \right) \equiv 1,
\end{equation*}
since $N$ is finite almost surely and the convergence of the speed to $S$ occurs regardless the initial condition, as long as it is finite. Taking the expected value on both sides finishes the proof.
\end{proof}

\subsection{A Counter-example}\label{s:counterexamples}
We conclude the discussion about the non-identically distributed case with the construction of a ballistic TBRW whose speed at time $n$ does not converge to any (random) limit speed. The example illustrates that even under the hypothesis of independence of the sequence $\xi$, we do have relatively simple examples of the TBRW that can keep slowing down and speeding up infinitely often.

More formally, we will construct an independent sequence of random variables $\xi_* = \{\xi_n \}_n$ over the natural numbers such that a TBRW generated using $\xi_*$ is ballistic, meaning 
$$
    \liminf_{n \to \infty}\frac{\dist{X_n}{o}}{n} > 0,
$$
$\Pl{\xi}$-almost surely, but the limit does not exist. The idea behind our example is to construct a TBRW that keeps alternating between two different speeds. Towards this goal, our example will be a TBRW that in some time window behaves as a Bernoulli instance with parameter $p$, and in another time window it  behaves as a Bernoulli instance with another parameter $q$. The main technical difficulty here is the choice of the time window, which needs to be chosen with some uniformity in mind.

In order to construct our example, we first observe that, since the speed, $v(p)$, of the TBRW generated by an i.i.d Bernoulli sequence of parameter $p$, is continuous by Theorem \ref{t:1}, we can fix $p,q \in (0,1]$ such that $0<v(p) < v(q)$. To avoid clutter, we will introduce the following shorthand for the distance
\begin{equation}\label{def:dist}
    D_m := \dist{X_m}{o}.
\end{equation}
And we will fix a finite initial state $(T,x)$. The next step is to construct the leaf process $\xi_*$. The construction is done inductively. We start by putting $k_0 = 0$ and defining $k_1$ to be the following natural number
\begin{equation}
    k_1 := \inf\Lbrace n > k_0 \; : \; \left| \El{T,x;p}\left[\frac{D_m}{m}\right] - v(p) \right| < \frac{1}{2}, \forall m \ge n \Rbrace.
\end{equation}
Notice that Theorem \ref{t:lln} together with the Bounded Convergence Theorem imply that for any finite initial condition $(T,x)$,
$$
\lim_{m \to \infty }\El{T,x;p}\left[\frac{D_m}{m}\right] = v(p).
$$
This implies the existence of $k_1$. We then consider $k_1$ independent Bernoulli random variables of parameter $p$, $\xi_1,\xi_2,\dots, \xi_{k_1}$ as the first $k_1$ variables of our sequence~$\xi_*$. 

Having defined $\xi_1,\xi_2,\dots, \xi_{k_1}$, we can define the first $k_1$ steps of our TBRW starting from $(T,x)$ and using $\xi_1,\xi_2,\dots, \xi_{k_1}$ to add vertices. We will denote these $k_1$ steps by $\{(T^*_j,X_j^*)\}_{j=0}^{k_1}$. Now, we define $k_2$ as 
\begin{equation}
    k_2 := \inf \Lbrace n > 0 \; : \; \left| \El{T^*_{k_1},X^*_{k_1};q}\left[\frac{D_m}{m}\right] - v(q) \right| < \frac{1}{2\cdot2},  \forall m \ge n\Rbrace + 2k_1
\end{equation}
and set $\xi_{k_1+1} = \dots = \xi_{k_2} = {\rm Ber}(q)$. 

Before moving to the general step, let us say a few words about $k_2$. We highlight that  $\El{T^*_{k_1},X^*_{k_1};q}$ denotes the expectation of a Bernoulli instance of the TBRW with parameter $q$ starting from the random initial condition $(T^*_{k_1},X^*_{k_1})$. In other words,  $\El{T^*_{k_1},X^*_{k_1};q}\left[\frac{D_m}{m}\right]$ is actually a random variable. We first sample the initial state by letting our TBRW to run for $k_1$ steps, which is possible since we have constructed the first $k_1$ $\xi$'s. Then we simply run a Bernoulli instance of the TBRW with parameter $q$. Since we have finitely many possible states for $(T^*_{k_1},X^*_{k_1})$ and for any finite initial state, $D_m/m$ converges to $v(q)$ under these settings, $k_2$ is a finite number.

The general case is done by induction as follows. Assuming we have defined $\xi_n$ for any $n \le k_{j-1}$, if $j$ is even, we let $k_j$ be
\begin{equation}
     k_j := \inf \Lbrace n > 0 \; : \; \left| \El{T^*_{k_{j-1}},X^*_{k_{j-1}};q}\left[\frac{D_m}{m}\right] - v(q) \right| < \frac{1}{2j},  \forall m \ge n\Rbrace + jk_{j-1}
\end{equation}
and put $\xi_{k_{j-1}+1} = \dots = \xi_{k_j} = {\rm Ber}(q)$. If $j$ is odd, $k_j$ becomes
\begin{equation}
     k_j := \inf \Lbrace n > 0 \; : \; \left| \El{T^*_{k_{j-1}},X^*_{k_{j-1}};p}\left[\frac{D_m}{m}\right] - v(p) \right| < \frac{1}{2j},  \forall m \ge n\Rbrace + jk_{j-1}
\end{equation}
and we set $\xi_{k_{j-1}+1} = \dots = \xi_{k_j} = {\rm Ber}(p)$. This concludes the construction of $\xi_*$.

Now we are left with showing that a TBRW starting from $(T,x)$ and with sequence $\xi_*$ does not have a well-defined speed. Before we show how this is done, let us make some important observations about our TBRW generated by $\xi_*$ and the sequence $(k_j)_j$. Firstly, it follows by construction that 
\begin{equation}\label{eq:limkj}
    \lim_{j \to \infty }\frac{k_{j-1}}{k_j} = 0.
\end{equation}
Secondly, if we shift the process by $k_{j-1}$, with $j$ even, by the construction and Simple Markov proeprty, it follows that
\begin{equation}\label{eq:smp}
    \begin{split}
        \El{T,x; \xi_*}\left[D_{k_j} - D_{k_{j-1}} \right] & = \El{T,x; \xi^*}\left[\El{T^*_{k_{j-1}},X^*_{k_{j-1}}; \xi_*^{(k_{j-1})}}\left[D_{k_j - k_{j-1}} \right]\right] \\
        & =  \El{T,x; \xi^*}\left[\El{T^*_{k_{j-1}},X^*_{k_{j-1}}; q}\left[D_{k_j - k_{j-1}} \right]\right],
    \end{split}
\end{equation}
since by construction $\xi_{k_{j-1}+1} = \dots = \xi_{k_j} = {\rm Ber}(q)$, which implies that after shifting by $k_{j-1}$, by $k_j - k_{j-1}$ steps we observe a Bernoulli instance of the TBRW of parameter $q$ and initial state $(T^*_{k_{j-1}},X^*_{k_{j-1}})$. The exact same arguments holds for $j$ odd, replacing $q$ by $p$.

With all the above in mind, we will finally show that the process we have just constructed does not have a well-defined speed. In order to do that, suppose towards a contradiction that there exists a positive random variable $Y$ such that 
\begin{equation}
    \lim_{m \to \infty }\frac{D_m}{m} = Y,
\end{equation}
$\Pl{T,x;\xi^*}$-almost surely. Then, by the Bounded Convergence Theorem
\begin{equation}\label{eq:convmean}
    \lim_{m \to \infty }\El{T,x; \xi^*}\left[\frac{D_m}{m} \right]= \El{T,x; \xi^*}Y.
\end{equation}
However, for any $j$ even, using \eqref{eq:limkj} and \eqref{eq:smp} together with Jensen's inequality
\begin{equation*}
    \begin{split}
      \left |  \El{T,x; \xi^*}\left[\frac{D_{j}}{k_{j}} \right] - v(q) \right| & = \left |\El{T,x; \xi^*}\left[\frac{D_{k_j} - D_{k_{j-1}}}{k_{j}} \right] -v(q)  + \El{T,x; \xi^*}\left[\frac{D_{k_{j-1}}}{k_{j}} \right] \right| \\ 
      & \le \left | \El{T,x; \xi^*}\left[\El{T^*_{k_{j-1}},X^*_{k_{j-1}}; q}\left[\frac{D_{k_j - k_{j-1}}}{k_j - k_{j-1}} \right]\frac{k_j - k_{j-1}}{k_j}\right] - v(q) \right |\\
      & \quad + \frac{k_{j-1}}{k_j} \\
      & \le \frac{1}{2j} + \frac{2k_{j-1}}{k_j}.
    \end{split}
\end{equation*}
Arguing similarly for $j$ odd, we conclude that the sequence of real numbers given by $(\El{T,x; \xi^*}\left[\frac{D_m}{m} \right])_m$ has two subsequences converging to two different values: $v(p)$ and $v(q)$, which contradicts \eqref{eq:convmean}. This contradiction shows that $D_m/m$ cannot converge almost surely.

\section{Final Comments}\label{s:finalc}
We conclude this work discussing about some interesting problems for future research. These problems illustrate how versatile the model is and that there is a lot of different perspectives one might want to explore.

\subsubsection*{Speed in the Bernoulli case.} As we have said before, in personal communication, Y. Peres suggested the investigation of the limit speed function in the Bernoulli case. This instance already leads to interesting questions involving analytical properties of $v = v(p)$, for example: (1) {\it is $v$ differentiable?} (2) {\it is $v$ increasing in $p$?}

After a constructive discussion with Professor F. Rassoul-Agha about $v$, who kindly shared the graph of his simulations, we were inspired to conduct our own simulation. The simulations suggest that indeed $v$ is strictly increasing on $p$ and that it is concave up with a possibly zero derivative at $p=0$, see Figure \ref{fig:plot} below.
\begin{figure}[h!]
    \centering
    \includegraphics[width=0.8\linewidth]{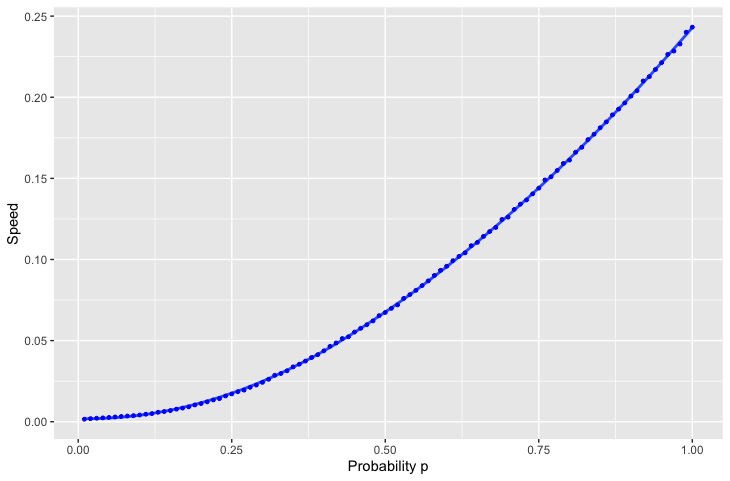}
    \caption{$1000$ iterations of $10000$ steps of the TBRW for different values of $p$.}
    \label{fig:plot}
\end{figure}

The code is available at \url{https://github.com/rbribeiro/tbrw} as an $R$ package.

\subsubsection*{Transient but not ballistic regimes.} If one desires to observe non-ballistic regimes, then condition \eqref{def:UE} must be dropped. One way to drop it is considering leaf process such that
$$
\lim_{n \to \infty }P(\xi_n \ge 1) = 0.
$$
That is, considering instances of the TBRW whose probability of adding one leaf goes to zero as time goes to infinity. In \cite{EIR21} the authors showed that a TBRW, with independent leaf process given by $\xi_n \sim {\rm Ber}(n^{-\gamma)}$, is recurrent if $\gamma>1/2$. The conjecture is that the model undergoes a phase transition at $\gamma = 1/2$. That is, for $\gamma \in (0,1/2)$ the walker is transient. Notice that in this regime, the walker cannot be ballistic, since $\dist{X_n}{o} \le \sum_{i=1}^n \xi_i+|T_0|$ almost surely, which is of order $n^{1-\gamma}$.

The regime below $1/2$ appears to be a significantly more challenging problem. When $\gamma > 1/2$, the recurrence is connected to the mixing properties of the environment. In this higher regime, the probability of adding a new vertex decreases fast enough that the walker has time to mix before attaching a new vertex. However, this reasoning does not hold when $\gamma < 1/2$. Also, in the lower regime, the probability of adding a new vertex is not high enough for one to extend the arguments applied to the case when the probability is $p$. This indicates that novel approaches are needed for analyzing this regime.

\subsubsection*{Dropping independence on the leaf process.} As we discussed at Section \ref{s:noniid}, even under the hypothesis of independence of the leaf process, the TBRW might not have a well-defined limit speed. Thus, if one wants to drop independence, some sort of mixing condition must be imposed on the leaf process in order to guarantee a LLN. The question is then what sort of mixing condition must be imposed. Perhaps some condition similar, at least in spirit, to the one investigated by E. Guerra, G. Valle and E. Vares in \cite{GVV22} in the context of RWRE. Note that if the identically distributed hypothesis over the leaf process is dropped, it is not clear if there exists a renewal structure similar to the one we provide in Theorem \ref{t:renew}. 

\section*{Acknowledgments}
The author would like to thank Caio Alves (Oak Ridge National Lab) for his invaluable comments and suggestions, and also for his long-term friendship. The author also would like to thank Professor Firas Rassoul-Agha (Univ. of Utah) for suggesting and sharing the results of his simulations. The author also would like to thank Professor Roberto I. Oliveira (IMPA) for suggesting the generalization of the continuity to general distributions and his comments on the first versions of this manuscript and Professor Alejandro Ram\'irez (NYU Shanghai) for reading the first version of this manuscript.

\bibliographystyle{plain}
\bibliography{ref}

\end{document}